\documentclass[12pt]{amsart}
\usepackage{graphicx}
\usepackage{amsmath}
\usepackage{amsthm}
\usepackage{amssymb}
\usepackage{amsfonts}
\usepackage{mathtools}
\usepackage{xcolor}
\usepackage{hyperref}
\usepackage{mathrsfs}
\usepackage{tikz}
\usepackage{xstring}
\usetikzlibrary{calc}
\hypersetup{
	colorlinks,
	citecolor=black,
	filecolor=black,
	linkcolor=black,
	urlcolor=black
}
\usepackage[capitalize, nameinlink, noabbrev, nosort]{cleveref}
\usepackage{float}
\usepackage{fullpage}
\usepackage{bbm}
\usepackage{bm}
\usepackage{ytableau}
\usepackage{diagbox}
\ytableausetup{boxsize=2em}

\newcommand{\rkey}[1]{{\normalfont\sffamily\bfseries\footnotesize[#1]}}      
\newcommand{\cjkey}[1]{{\normalfont\sffamily\bfseries\footnotesize\color{red}[#1\textsuperscript{\dag}]}} 
\newcommand{\cj}[1]{\textcolor{red}{#1}}                                      

\theoremstyle{plain}
\newtheorem{theorem}{Theorem}[section]
\newtheorem{proposition}[theorem]{Proposition}

\newtheorem{corollary}[theorem]{Corollary}
\newtheorem{lemma}[theorem]{Lemma}

\theoremstyle{definition}
\newtheorem{definition}[theorem]{Definition}
\newtheorem{example}[theorem]{Example}
\newtheorem{remark}[theorem]{Remark}

\DeclareMathOperator{\SYBT}{SYBT}
\DeclareMathOperator{\shape}{shape}
\DeclareMathOperator{\Des}{Des}
\DeclareMathOperator{\fmaj}{fmaj}
\DeclareMathOperator{\Res}{Res}
\DeclareMathOperator{\Frob}{Frob}
\DeclareMathOperator{\Hilb}{Hilb}
\DeclareMathOperator{\Ind}{Ind}
\DeclareMathOperator{\CT}{CT}
\DeclareMathOperator{\alt}{alt}
\DeclareMathOperator{\GL}{GL}
\DeclareMathOperator{\Sym}{Sym}
\DeclareMathOperator{\sgn}{sgn}
\DeclareMathOperator{\reg}{reg}
\DeclareMathOperator{\ch}{char}
\DeclareMathOperator{\ad}{ad}

\newcommand{\Z}{\mathbb{Z}}
\newcommand{\C}{\mathbb{C}}
\newcommand{\x}{\mathbf{x}}
\newcommand{\y}{\mathbf{y}}
\newcommand{\z}{\mathbf{z}}

\newcommand{\motzkin}[2]{
    \begin{tikzpicture}[scale=0.6, baseline=(current bounding box.center)]
        \coordinate (current) at (0,0);
        \ifcsname c@step\endcsname
        \else
            \newcounter{step}
        \fi
        \setcounter{step}{1}
        \foreach \step in {#1} {
            \ifnum\pdfstrcmp{\step}{up}=0
                \draw[thick] (current) -- ++(1,1);
                \coordinate (current) at ($(current) + (1,1)$);
            \fi
            \ifnum\pdfstrcmp{\step}{flat-theta}=0
                \draw[thick] (current) -- ++(1,0) node[midway, above, font=\footnotesize] {$\theta_{\arabic{step}}$};
                \coordinate (current) at ($(current) + (1,0)$);
            \fi
            \ifnum\pdfstrcmp{\step}{flat-xi}=0
                \draw[thick] (current) -- ++(1,0) node[midway, above, font=\footnotesize] {$\xi_{\arabic{step}}$};
                \coordinate (current) at ($(current) + (1,0)$);
            \fi
            \ifnum\pdfstrcmp{\step}{flat}=0
                \draw[thick] (current) -- ++(1,0);
                \coordinate (current) at ($(current) + (1,0)$);
            \fi
            \ifnum\pdfstrcmp{\step}{down}=0
                \draw[thick] (current) -- ++(1,-1) node[midway, above, font=\footnotesize, xshift=2pt, yshift=1pt] {$\theta_{\arabic{step}} \xi_{\arabic{step}}$};
                \coordinate (current) at ($(current) + (1,-1)$);
            \fi
            \stepcounter{step}
        }
        \node[anchor=north] at (0.5*\arabic{step}, -0.5) {\normalsize #2};
    \end{tikzpicture}
}

\title{Type $B$ fermionic coinvariant rings}

\author[Y. Jiang]{Yuhan Jiang}
\address{Department of Mathematics\\
         University of California, Berkeley, CA, USA}
\email{jyh@math.berkeley.edu}

\author[J. Lentfer]{John Lentfer}
\address{Department of Mathematics\\
         University of California, San Diego, CA, USA}
\email{jlentfer@ucsd.edu}

\begin{document}

\begin{abstract}
    Let $\mathfrak{B}_n$ denote the hyperoctahedral group.
    The type $B$ coinvariant rings $R_{\mathfrak{B}_n}^{(k,j)}$ are quotients of the ring of polynomials in $k$ sets of $n$ commuting variables and $j$ sets of $n$ anticommuting variables by the ideal generated by the diagonal $\mathfrak{B}_n$-invariants without constant term.
    Building upon the work of Kim--Rhoades (2022), we give an explicit formula for the bigraded Frobenius series of $R_{\mathfrak{B}_n}^{(0,2)}$: the bigraded multiplicity of each irreducible $\mathfrak{B}_n$-character is a single Schur polynomial, so $R_{\mathfrak{B}_n}^{(0,2)}$ is multiplicity-free as a $\operatorname{GL}_2 \times \mathfrak{B}_n$-module.
    We then determine that the trigraded multiplicity of the sign character of $R_{\mathfrak{B}_n}^{(0,3)}$ is given by a single Schur function.
    Finally, for all $k$ and $j$, we determine the multiplicity of the standard character in the type $A$ coinvariant ring $R_{n}^{(k,j)}$, as well as the multiplicities of the characters indexed by the bipartitions $((n-1),(1))$ and $((n-1,1),\varnothing)$ in $R_{\mathfrak{B}_n}^{(k,j)}$.
    These are the first nontrivial characters established for all $(k,j)$ in either of types $A$ or $B$.
\end{abstract}

\maketitle

\section{Introduction}\label{sec:introduction}

The classical coinvariant ring $R_n^{(1,0)}$ and the type $B$ coinvariant ring $R_{\mathfrak{B}_n}^{(1,0)}$ are important mathematical objects, appearing throughout combinatorics and algebra.
They are isomorphic to the cohomology rings of the type $A$ and type $B$ flag varieties, respectively \cite{Borel,Leray}.

The diagonal coinvariant ring $R_n^{(2,0)}$ and its type $B$ version $R_{\mathfrak{B}_n}^{(2,0)}$ were introduced by Haiman \cite{Haiman1994}, who later determined the dimension, $(n+1)^{n-1}$, and the bigraded Frobenius series of $R_n^{(2,0)}$ \cite{Haiman2002}. 
In type $B$, the exact dimension of $R_{\mathfrak{B}_n}^{(2,0)}$ is unknown. 
Gordon \cite{Gordon} established a lower bound of $(2n+1)^n$ by exhibiting a quotient module of that dimension, confirming a conjecture of Haiman \cite{Haiman1994}. Ajila--Griffeth \cite{AjilaGriffeth} improved the bound further.

Several other coinvariant ring variations have since been studied in types $A$ and $B$.
The diagonal fermionic coinvariant rings $R_n^{(0,2)}$ and $R_{\mathfrak{B}_n}^{(0,2)}$ were studied in a type-independent way by Kim--Rhoades \cite{MR4381935}.
The superspace coinvariant ring $R_n^{(1,1)}$ was studied in \cite{Angarone2025, fields2025, RhoadesWilson2023, SwansonWallach1, SwansonWallachII} and the type $B$ version $R_{\mathfrak{B}_n}^{(1,1)}$ in \cite{bhattacharya, RhoadesBhattacharya, MR4674564, SwansonWallach1, SwansonWallachII}.

For a summary of the conjectures and results on the dimension of $R_n^{(k,j)}$ and $R_{\mathfrak{B}_n}^{(k,j)}$ for small values of $k,j$, see Tables~\ref{tab:dim} and~\ref{tab:dim-B_n}; 
for the multiplicity of the sign character, see Tables~\ref{tab:sign} and~\ref{tab:sign-B_n}.

While the type $A$ coinvariant rings are typically studied first, there is a strong interest in studying the type $B$ coinvariant rings, especially to determine when the two can be unified in a type-independent way and when important differences arise.

The rings discussed so far are all special cases of the following general construction (see \cite{Bergeron2013, BergeronOPAC, Bergeron2020, Lentfer-Supersymmetry}).
Define the \emph{$(k,j)$-bosonic-fermionic coinvariant ring for a finite group $G \subset \GL_n$} by
\begin{equation}\label{eq:coinv} R_G^{(k,j)} = \C[\bm{x}^{(1)}, \ldots, \bm{x}^{(k)}, \bm{\theta}^{(1)}, \ldots, \bm{\theta}^{(j)}]/\langle \C[\bm{x}^{(1)}, \ldots, \bm{x}^{(k)}, \bm{\theta}^{(1)}, \ldots, \bm{\theta}^{(j)}]_+^{G} \rangle,\end{equation}
where $\C[\bm{x}^{(1)}, \ldots, \bm{x}^{(k)}, \bm{\theta}^{(1)}, \ldots, \bm{\theta}^{(j)}]^G_+$ denotes the $G$-invariant polynomials without constant term, under the diagonal action of $G$ on each set of variables.
The sets of bosonic variables are denoted by $\bm{x}^{(i)} = \{ x^{(i)}_1,\ldots,x^{(i)}_n \}$ and the sets of fermionic variables are denoted by $\bm{\theta}^{(i)} = \{ \theta^{(i)}_1,\ldots,\theta^{(i)}_n \}$.
Note that $\C[\bm{x}^{(1)},\ldots,\bm{x}^{(k)},
     \bm{\theta}^{(1)},\ldots,\bm{\theta}^{(j)}] \cong
   \Sym\big((\C^n)^{\oplus k}\big) \otimes
   {\textstyle\bigwedge}\big((\C^n)^{\oplus j}\big)$.
The bosonic variables $\bm{x}^{(\ell)}$ commute with all variables, while the fermionic variables $\bm{\theta}^{(\ell)}$ anticommute with one another.
In particular, any fermionic variable squared is $0$.

When $G$ is the symmetric group $\mathfrak{S}_n$, acting diagonally by permuting the indices within each set of variables, we denote this by $R_n^{(k,j)}$.
When $G$ is the hyperoctahedral group $\mathfrak{B}_n$, acting diagonally by signed permutations of the indices within each set, we denote this by $R_{\mathfrak{B}_n}^{(k,j)}$.

In this paper, we primarily study the type $B$ fermionic coinvariant rings $R_{\mathfrak{B}_n}^{(0,j)}$.
Section~\ref{sec:background} reviews background material; as a warm-up, Section~\ref{sec:01} determines the Frobenius series of $R_{\mathfrak{B}_n}^{(0,1)}$ (Proposition~\ref{prop:one-set}), which encodes its $\mathfrak{B}_n$-module structure.
Sections~\ref{sec:02}--\ref{sec:standard} contain our main contributions, summarized below. 
Appendix~\ref{sec:harmonic} develops the harmonic space $H_{\mathfrak{B}_n}^{(k,j)}$, which is isomorphic to $R_{\mathfrak{B}_n}^{(k,j)}$ as a multigraded $\mathfrak{B}_n$-module, together with the polarization operators acting on it. 
Appendix~\ref{sec:appendix} gives data for the Frobenius series and dimension of $R_{\mathfrak{B}_n}^{(0,3)}$ up through $n=4$.
Our main contributions are as follows:
\begin{itemize}
    \item Building on the work of Kim--Rhoades \cite{MR4381935}, we give an explicit formula for the bigraded multiplicity of every irreducible character in $R_{\mathfrak{B}_n}^{(0,2)}$, and hence for its bigraded Frobenius series (Theorem~\ref{thm:explicit-0-2}):
    \begin{equation}
    \Frob(R_{\mathfrak{B}_n}^{(0,2)};u,v) = \sum_{\substack{(\lambda, \mu)\vdash n,\\ \lambda = (\lambda_1,\lambda_2),\\ \mu = (2^\ell, 1^{n-\lambda_1-\lambda_2-2\ell}) }} s_{(n-\ell-\lambda_1, \ell+\lambda_2)}(u,v)s_{\lambda}(\x) s_{\mu}(\y).
    \end{equation}
    As a consequence, we obtain a multiplicity-free decomposition of $R_{\mathfrak{B}_n}^{(0,2)}$ as a $\GL_2 \times \mathfrak{B}_n$-module (Corollary~\ref{cor:multiplicity-free}):
    \begin{equation}
        R_{\mathfrak{B}_n}^{(0,2)} \cong \bigoplus_{\substack{(\lambda, \mu) \vdash n,\\
        \lambda = (\lambda_1,\lambda_2),\\
        \mu = (2^\ell, 1^{n-\lambda_1-\lambda_2 - 2\ell})}}
        \mathbb{S}^{(n-\ell-\lambda_1, \ell+\lambda_2)}(\C^2) \otimes V^{(\lambda, \mu)}.
    \end{equation}
    We also recover the modified Motzkin path formula for the Hilbert series of $R_{\mathfrak{B}_n}^{(0,2)}$ of Kim--Rhoades (Proposition~\ref{prop:KR-Hilb}) and obtain a new Hilbert series formula in terms of Kostka numbers (Corollary~\ref{cor:new-Hilbert-B02}).
    \item We determine the trigraded multiplicity of the sign character of $R_{\mathfrak{B}_n}^{(0,3)}$ (Theorem~\ref{thm:03sgn}), paralleling work in \cite{lentfer2026signcharactertriagonalfermionic} for the symmetric group:
    \begin{equation}
        \langle \Frob(R_{\mathfrak{B}_n}^{(0,3)};u,v,w),s_{\varnothing}(\x) s_{(1^n)}(\y) \rangle =s_{(n)}(u,v,w).
    \end{equation}
    \item We determine the multiplicity of the standard character $\chi^{(n-1,1)}$ in $R_n^{(k,j)}$ for all $k,j$ (Theorem~\ref{thm:standard-character}):
    \begin{equation}
        \langle \Frob(R_n^{(k,j)}; \mathbf{q};\mathbf{u}),
        s_{(n-1,1)}(\z) \rangle
        = s_{(1)}(\mathbf{q}/\mathbf{u}) + s_{(2)}(\mathbf{q}/\mathbf{u})
        + \cdots + s_{(n-1)}(\mathbf{q}/\mathbf{u}).
    \end{equation}
    \item We determine the multiplicity of the characters $\chi^{((n-1,1),\varnothing)}$ and $\chi^{((n-1),(1))}$ in $R_{\mathfrak{B}_n}^{(k,j)}$ for all $k,j$ (Theorem~\ref{thm:standard-character-B_n}):
    \begin{equation}
        \langle \Frob(R_{\mathfrak{B}_n}^{(k,j)}; \mathbf{q};\mathbf{u}),
        s_{(n-1)}(\x)s_{(1)}(\y) \rangle
        = s_{(1)}(\mathbf{q}/\mathbf{u}) + s_{(3)}(\mathbf{q}/\mathbf{u})
        + \cdots + s_{(2n-1)}(\mathbf{q}/\mathbf{u}),
    \end{equation}
    and
        \begin{equation}
        \langle \Frob(R_{\mathfrak{B}_n}^{(k,j)}; \mathbf{q};\mathbf{u}),
        s_{(n-1,1)}(\x)s_{\varnothing}(\y) \rangle
        = s_{(2)}(\mathbf{q}/\mathbf{u}) + s_{(4)}(\mathbf{q}/\mathbf{u})
        + \cdots + s_{(2n-2)}(\mathbf{q}/\mathbf{u}).
    \end{equation}
\end{itemize}


\begin{table}[b]
\centering
\begin{tabular}{|c|ccc|}
\hline
\diagbox{$k$}{$j$} & $0$ & $1$ & $2$ \\ \hline
$0$ & $1$          & $2^{n-1}$\,\rkey{a}        & $\binom{2n-1}{n}$\,\rkey{b}     \\[3pt]
$1$ & $n!$\,\rkey{c} & $\sum_{i=1}^n i! S(n,i)$\,\rkey{d}    &  \cj{$2^{n-1}n!$} \,\cjkey{e}                                      \\[3pt]
$2$ & $(n+1)^{n-1}$\,\rkey{f} & \cj{$\sum_{i=0}^{n+1} \binom{n+1}{i} \frac{i^n}{2(n+1)}$}\,\cjkey{g}   &  ---              \\[3pt]
$3$ &  \cj{$2^n(n+1)^{n-2}$}\,\cjkey{h}               & ---                  & ---                 \\ \hline
\end{tabular}
\caption{Dimension of $R_{n}^{(k,j)}$.}
\label{tab:dim}
\smallskip

\begin{minipage}{\textwidth}\footnotesize
\emph{Notes.} Entries marked with a dagger~(\dag) (in red) are conjectural; all others are established.\par\smallskip
\rkey{a}~Easy. \quad
\rkey{b}~Conjectured by \cite{Zabrocki2020}; proven by \cite{MR4381935}.\quad
\rkey{c}~Proven by \cite{Artin}.\quad
\rkey{d}~$S(n,i)$ is the Stirling number of the second kind; conjectured by \cite{MR4674564}; proven by \cite{RhoadesWilson2023}.\quad
\cjkey{e}~Conjectured by \cite{Zabrocki2020, Bergeron2020}.\quad
\rkey{f}~Conjectured by \cite{Haiman1994}; proven by \cite{Haiman2002}.\quad
\cjkey{g}~Conjectured by \cite{Zabrocki2020, Bergeron2020}.\quad
\cjkey{h}~Conjectured by \cite{Haiman1994}.
\end{minipage}
\end{table}

\begin{table}
\centering
\begin{tabular}{|c|ccc|}
\hline
\diagbox{$k$}{$j$} & $0$ & $1$ & $2$ \\ \hline
$0$ & $1$          & $2^n$\,\rkey{a}        & $\binom{2n+1}{n}$\,\rkey{b}     \\[3pt]
$1$ & $2^n n!$\,\rkey{c} & $\sum_{i=0}^n (2i)!! S^B(n,i)$\,\rkey{d}    &    \cj{$4^n n!$}\,\cjkey{e}                                     \\[3pt]
$2$ & $\geq (2n+1)^{n}$\,\rkey{f} & --- &  ---              \\[3pt]
$3$ &    ---           & ---                  & ---                 \\ \hline
\end{tabular}
\caption{Dimension of $R_{\mathfrak{B}_n}^{(k,j)}$.}
\label{tab:dim-B_n}
\smallskip

\begin{minipage}{\textwidth}\footnotesize
\emph{Notes.} Entries marked with a dagger~(\dag) (in red) are conjectural; all others are established.\par\smallskip
\rkey{a}~Easy (see Proposition~\ref{prop:01-B_n-Hilb}).\quad
\rkey{b}~Proven by \cite{MR4381935}.\quad
\rkey{c}~Proven by \cite{MR59914, Chevalley} (see Theorem~\ref{thm:Chevalley-Shephard-Todd}).\quad
\rkey{d}~$S^B(n,i)$ is the type $B$ Stirling number of the second kind; conjectured by \cite{MR4674564}; proven by \cite{bhattacharya}.\quad
\cjkey{e}~Conjectured by \cite{Lentfer-12}.\quad
\rkey{f}~$(2n+1)^{n}$ was conjectured to be a lower bound in \cite{Haiman1994} and proven by \cite{Gordon}; the bound was improved by \cite{AjilaGriffeth}; exact value unknown.
\end{minipage}
\end{table}

\begin{table}
\centering
\begin{tabular}{|c|cccc|}
\hline
\diagbox{$k$}{$j$} & $0$ & $1$ & $2$ & $3$ \\ \hline
$0$ & $0$          & $1$\,\rkey{a}        & $n$\,\rkey{b}      & $n^2-n+1$\,\rkey{c} \\[3pt]
$1$ & $1$\,\rkey{d} & $2^{n-1}$\,\rkey{e}    & \cj{$3^{n-1}$}\,\cjkey{f}                     &       \cj{$\frac{1}{2}F_{3n}$}\,\cjkey{g}                     \\[3pt]
$2$ & $C_n$\,\rkey{h} & \cj{$\mathsf{s}(n)$}\,\cjkey{i} &  \cj{$2^{n-1}C_n$}\,\cjkey{j}    & ---                        \\[3pt]
$3$ &    \cj{$\frac{2}{n(n+1)} \binom{4n+1}{n-1}$}\,\cjkey{k}         & ---                  & ---                  & ---                        \\
\hline
\end{tabular}
\caption{Multiplicity of the sign character of $R_{n}^{(k,j)}$.}
\label{tab:sign}
\smallskip

\begin{minipage}{\textwidth}\footnotesize
\emph{Notes.} Entries marked with a dagger~(\dag) (in red) are conjectural; all others are established.\par\smallskip
\rkey{a}~Proven by \cite{HaglundSergel} (see Proposition~\ref{prop:HS-one-set}).\quad
\rkey{b}~Proven by \cite{MR4381935} (see Proposition~\ref{prop:02-all-characters}).\quad
\rkey{c}~Conjectured by \cite{Bergeron2020}; proven by \cite{lentfer2026signcharactertriagonalfermionic} (see Theorem~\ref{thm:A03-sign}). \quad
\rkey{d}~Classical \cite{Borel,Leray}.\quad
\rkey{e}~Proven by \cite{SwansonWallach1}.\quad
\cjkey{f}~Conjectured by \cite{Bergeron2020}.\quad
\cjkey{g}~$F_n$ denotes the Fibonacci number with initial conditions $F_1=F_2=1$; conjectured by \cite{Bergeron2020}.\quad
\rkey{h}~$C_n$ denotes the Catalan number; conjectured by \cite{Haiman1994}; proven by \cite{Haiman2002}.\quad
\cjkey{i}~$\mathsf{s}(n)$ denotes the little Schr\"oder number; conjectured by \cite{Bergeron2020}. \quad
\cjkey{j}~Conjectured by \cite{Bergeron2020}.\quad
\cjkey{k}~Conjectured by \cite{Haiman1994}.
\end{minipage}
\end{table}

\begin{table}
\centering
\begin{tabular}{|c|cccc|}
\hline
\diagbox{$k$}{$j$} & $0$ & $1$ & $2$ & $3$ \\ \hline
$0$ & $0$          & $1$\,\rkey{a}        & $n+1$\,\rkey{b}      & $\binom{n+2}{2}$\,\rkey{c} \\[3pt]
$1$ & $1$\,\rkey{d} & $2^{n}$\,\rkey{e}    &    \cj{$3^{n}$}\,\cjkey{f}                   &       \cj{$\frac{1}{2}F_{3(n+1)}$}\,\cjkey{g}                      \\[3pt]
$2$ & \cj{$\binom{2n}{n}$}\,\cjkey{h} & \cj{$D_n$}\,\cjkey{i} &    \cj{$2^n\binom{2n}{n}$}\,\cjkey{j}  & ---                        \\[3pt]
$3$ &      ---         & ---                  & ---                  & ---                        \\
\hline
\end{tabular}
\caption{Multiplicity of the sign character of $R_{\mathfrak{B}_n}^{(k,j)}$.}
\label{tab:sign-B_n}
\smallskip

\begin{minipage}{\textwidth}\footnotesize
\emph{Notes.} Entries marked with a dagger~(\dag) (in red) are conjectural; all others are established.\par\smallskip
\rkey{a}~Proposition~\ref{prop:one-set}.\quad
\rkey{b}~Corollary~\ref{cor:typeB-hooks}.\quad
\rkey{c}~Theorem~\ref{thm:03sgn}.\quad
\rkey{d}~Classical \cite{Borel,Leray} (see Theorem~\ref{thm:Stembridge}).\quad
\rkey{e}~Proven by \cite{SwansonWallach1}.\quad
\cjkey{f}~Based on computed data for $n \leq 4$.\quad
\cjkey{g}~$F_n$ denotes the Fibonacci number with initial conditions $F_1=F_2=1$; based on computed data for $n \leq 4$.\quad
\cjkey{h}~Conjectured by \cite{Haiman1994, Stump}; this has been established as a lower bound, following from \cite{Gordon}.\quad
\cjkey{i}~$D_n$ denotes the central Delannoy numbers; based on computed data for $n \leq 4$; this will be discussed in future work.\quad
\cjkey{j}~Based on computed data for $n \leq 4$.
\end{minipage}
\end{table}

\section{Background}\label{sec:background}

\subsection{\texorpdfstring{$q$}{q}-analogues}
Define the $q$-integers by $[r]_q = 1 + q + \cdots + q^{r-1}$.
Define the $q$-analogue of the double factorial by 
\begin{equation}
    [r]_q!! = [r]_q [r-2]_q [r-4]_q \cdots 
    \begin{cases}
        [2]_q & \text{if $r$ even,}\\
        [1]_q & \text{if $r$ odd.}
    \end{cases}
\end{equation}
For two variables $u,v$, define the $u,v$-integers by $[r]_{u,v} := u^{r-1} + u^{r-2} v + \cdots + u v^{r-2} + v^{r-1}$.
By convention, $[r]_{u,v} = 0$ if $r = 0$.
We will use the following expression for two-row Schur functions:
\begin{equation}\label{eq:two-row-schur}
    s_{(a,b)}(u,v) = (uv)^b [a-b+1]_{u,v}.
\end{equation}

\subsection{Hyperoctahedral group statistics}
Let $C_2$ denote the cyclic group with two elements. 
The hyperoctahedral group $\mathfrak{B}_n$ is the wreath product group $C_2 \wr \mathfrak{S}_n$. It is a signed permutation group and is the Weyl group of the root systems of types $B$ and $C$.

See for example \cite{MR3629266} for the following definitions.
A \emph{bipartition} $(\lambda,\mu)$ of $n$ is an ordered pair of two partitions such that $|\lambda| + |\mu| = n$.
We write this by $(\lambda, \mu) \vdash n$.

Given a bipartition $(\lambda,\mu)$, a \emph{standard Young bitableau} of shape $(\lambda,\mu)$ and of size $n = |\lambda| + |\mu|$ is a pair $Q = (Q^+, Q^-)$ of tableaux which are (strictly) increasing along rows and columns, and
\begin{itemize}
    \item $Q^+$ has shape $\lambda$;
    \item $Q^-$ has shape $\mu$;
    \item every element of $[n]$ appears (exactly once) in either $Q^+$ or $Q^-$.
\end{itemize}
Denote by $\SYBT_n$ the set of standard Young bitableaux of size $n$.

The \emph{descent set}, denoted by $\Des(Q)$, of a bitableau $Q = (Q^+,Q^-) \in \SYBT_n$ consists of all $i \in [n-1]$ such that 
\begin{itemize}
\item either $i$ and $i+1$ appear in the same tableau and $i+1$ appears in a strictly lower row than $i$,
\item or $i$ appears in $Q^+$ and $i+1$ appears in $Q^-$.
\end{itemize}
The \emph{flag-major index} of $Q \in \SYBT_n$ is 
\begin{equation}
\fmaj(Q) := 2 \left(\sum_{i \in \Des(Q)} i \right) + |\shape(Q^-)|.
\end{equation}

\subsection{The defining ideals}\label{subsec:defining-ideals}
Throughout, write $S = \C[\bm{x}^{(1)}, \ldots, \bm{x}^{(k)}, \bm{\theta}^{(1)},\ldots, \bm{\theta}^{(j)}]$, and for a finite group $G \subset \GL_n$ acting diagonally on $S$, write
\begin{equation}
    I_G^{(k,j)} := \langle S_+^{G} \rangle
\end{equation}
for the \emph{defining ideal} of equation~\eqref{eq:coinv}, so that $R_G^{(k,j)} = S/I_G^{(k,j)}$.

Let the \emph{polarized power sum} be defined by
\begin{equation}
\begin{aligned}
        p_{\bm{r};\bm{s}} :=& \
        p_{r_1,\ldots,r_k;s_1,\ldots,s_j}(\bm{x}^{(1)}_n,  \ldots, \bm{x}^{(k)}_n,\bm{\theta}_n^{(1)}, \ldots, \bm{\theta}_n^{(j)})\\ 
        =& \sum_{i=1}^n (x_i^{(1)})^{r_1} \cdots (x_i^{(k)})^{r_k} (\theta_i^{(1)})^{s_1} \cdots (\theta_i^{(j)})^{s_j}.
\end{aligned}
\end{equation}
We may add the subscript $n$ to $\bm{x}^{(i)}_n$ to record how many variables are in each set $\bm{x}^{(i)}$.

The following generating sets are results of Weyl \cite{Weyl} in the purely bosonic case $j = 0$. 
As noted in \cite{Zabrocki2019}, following from \cite{OrellanaZabrocki}, they extend to the bosonic-fermionic setting.
\begin{proposition}\label{prop:Weyl-bf}
    The algebra of diagonal invariants
    $\C[\bm{x}^{(1)}, \ldots, \bm{x}^{(k)}, \bm{\theta}^{(1)}, \ldots, \bm{\theta}^{(j)}]^{\mathfrak{S}_n}$ is generated by the polarized power sums
    \begin{equation}\label{eq:gen-set-A}
    \left\{ p_{\bm{r};\bm{s}} \;\middle|\;
    \begin{array}{l}
        1 \leq r_1 + \cdots + r_k + s_1 + \cdots + s_j \leq n, \\
        r_1,\ldots,r_k \in \Z_{\geq 0},\ s_1,\ldots, s_j \in \{0,1\}
    \end{array} \right\}.
\end{equation}
    Consequently, the set~\eqref{eq:gen-set-A} generates the defining ideal $I_{n}^{(k,j)}$.
\end{proposition}

\begin{proposition}\label{prop:Weyl-B_n-bf}
    The algebra of diagonal invariants
    $\C[\bm{x}^{(1)}, \ldots, \bm{x}^{(k)}, \bm{\theta}^{(1)}, \ldots, \bm{\theta}^{(j)}]^{\mathfrak{B}_n}$ is generated by the polarized power sums
    \begin{equation}\label{eq:gen-set-B}
    \left\{ p_{\bm{r};\bm{s}} \;\middle|\; 
    \begin{array}{l}
        2 \leq r_1 + \cdots + r_k +s_1 + \cdots +s_j \leq 2n, \\
        r_1 + \cdots + r_k +s_1 + \cdots +s_j\text{ even}, \\
        r_1,\ldots,r_k \in \Z_{\geq 0}, s_1,\ldots, s_j \in \{0,1\}
    \end{array}
    \right\}.
    \end{equation}
    Consequently, the set~\eqref{eq:gen-set-B} generates the defining ideal $I_{\mathfrak{B}_n}^{(k,j)}$.
\end{proposition}

Since $G$ preserves each multihomogeneous component of $S$, the invariant space $S^G$, and hence $S_+^G$, is spanned by multihomogeneous elements. 
Therefore $I_G^{(k,j)}$ is multihomogeneous and $R_G^{(k,j)}$ inherits the multigrading of $S$.

\begin{remark}\label{rem:averaging}
    Since $G$ is finite and we work over $\C$, every $G$-invariant element of $R_G^{(k,j)}$ lifts to a $G$-invariant polynomial: if $[f] \in R_G^{(k,j)}$ is invariant, then $g \cdot f \equiv f \pmod{I_G^{(k,j)}}$ for all $g \in G$, so the average $\frac{1}{|G|}\sum_{g \in G} g \cdot f$ is a $G$-invariant polynomial whose image in $R_G^{(k,j)}$ is $[f]$.
    (Note that this is an application of the Reynolds operator.)
    The same argument applies to the $H$-invariant elements for any subgroup $H \leq G$, since $I_G^{(k,j)}$ is $H$-stable.
    Since $I_G^{(k,j)}$ is multihomogeneous, the argument applies multidegree by multidegree.
    This Reynolds operator argument is used in the proofs of Lemma~\ref{lem:trivial} and Proposition~\ref{prop:parabolic-invariants}.
\end{remark}

\subsection{Hilbert and Frobenius series}
The ring $R_G^{(k,j)}$ is multigraded by:
\begin{equation}\label{eq:R-grading-Hilbert}
R_G^{(k,j)} = \bigoplus_{r_1,\ldots,r_k; s_1, \ldots, s_j \geq 0} (R_G^{(k,j)})_{r_1,\ldots,r_k; s_1, \ldots, s_j},
\end{equation}
where $r_i$ is the degree of the variables $\bm{x}^{(i)}$ and $s_i$ is the degree of the variables $\bm{\theta}^{(i)}$.
Thus its \emph{(multigraded) Hilbert series} is
\begin{equation} 
\Hilb(R_G^{(k,j)}; \mathbf{q}; \mathbf{u}) := \sum_{r_1,\ldots,r_k; s_1, \ldots, s_j \geq 0} \dim\left((R_G^{(k,j)})_{r_1,\ldots,r_k; s_1, \ldots, s_j} \right)q_1^{r_1} \cdots q_k^{r_k}u_1^{s_1} \cdots u_j^{s_j}.
\end{equation}
Here, $\mathbf{q} = (q_1,\ldots,q_k)$ and $\mathbf{u} = (u_1,\ldots,u_j)$. 
When $k \leq 2$, we may use $q = q_1$ and $t = q_2$. 
When $j \leq 3$, we may use $u = u_1$, $v = u_2$, and $w = u_3$.
Also, when $j \leq 3$ we write $\bm{\theta}, \bm{\xi}, \bm{\rho}$  for $\bm{\theta}^{(1)},  \bm{\theta}^{(2)}, \bm{\theta}^{(3)}$ respectively.

Let $\ch(M)$ denote the character of a $G$-module $M$. 
The irreducible representations of $\mathfrak{S}_n$ are indexed by partitions $\nu \vdash n$; we denote them by $V^\nu$. 
For a fixed $n$, the Frobenius characteristic map $F$ is an isomorphism between the group of virtual $\mathfrak{S}_n$-characters and the degree $n$ component of the ring of symmetric functions $\Lambda_\Z^n(\z)$, given by $F(\chi^\mu) = s_\mu(\z)$, where $\chi^\mu$ is the irreducible $\mathfrak{S}_n$-character indexed by $\mu$, and $s_\mu(\z)$ is a Schur function (see \cite[Section 7.18]{StanleyEC2}). 

Regarding $R_n^{(k,j)}$, we can apply the Frobenius characteristic map $F$ to the character of each multigraded component, and obtain the \emph{multigraded Frobenius series:} 
\begin{equation} \Frob(R_n^{(k,j)}; \mathbf{q}; \mathbf{u}) := \sum_{r_1,\ldots,r_k; s_1, \ldots, s_j \geq 0} F\ch\left((R_n^{(k,j)})_{r_1,\ldots,r_k; s_1, \ldots, s_j} \right)q_1^{r_1} \cdots q_k^{r_k}u_1^{s_1} \cdots u_j^{s_j}.\end{equation}

Next we will briefly overview some of the basics of representation theory of $\mathfrak{B}_n$ so that we can define the type $B$ Frobenius series.
The representation theory of $\mathfrak{B}_n$ is well understood (see \cite[Section 5.5]{GeckPfeiffer} for a type $B$ exposition; see \cite[Chapter I, Appendix B]{Macdonald} for wreath product groups more generally). 
Since it is a finite group, up to isomorphism, all of its representations may be decomposed into a direct sum of irreducible representations.
The irreducible representations of $\mathfrak{B}_n$ are indexed by bipartitions $(\lambda, \mu) \vdash n$; we denote them by $V^{(\lambda, \mu)}$. 
Let $\chi^{(\lambda, \mu)}$ denote the associated irreducible character.

For an element $g \in \mathfrak{B}_n$, we may write it as $g = (\epsilon, \sigma)$, where $\sigma \in \mathfrak{S}_n$ and $\epsilon = (\epsilon_1,\ldots,\epsilon_n) \in \{-1,+1\}^n$. 
The action is given by $g \cdot x_i^{(\ell)} = \epsilon_i x_{\sigma (i)}^{(\ell)}$ and $g \cdot \theta_i^{(\ell)} = \epsilon_i \theta_{\sigma (i)}^{(\ell)}$.
The \emph{determinant character} $\chi^{(\varnothing, (1^n))}$ is given by  $\det(g) = \sgn(\sigma) \prod_{i=1}^n \epsilon_i$.
By \emph{sign representation}, we mean the representation corresponding to the determinant character, that is, the sign representation of $\mathfrak{B}_n$ is $V^{(\varnothing, (1^n))} \cong \varepsilon_{\mathfrak{B}_n}$. 

The \emph{trivial representation} of $\mathfrak{B}_n$ is $V^{((n),\varnothing)} \cong \mathbbm{1}_{\mathfrak{B}_n}$. 
The defining representation of $\mathfrak{B}_n$ is isomorphic to $V^{((n-1),(1))}$. 
Then the exterior power $\wedge^i V^{((n-1),(1))} \cong V^{((n-i),(1^i))}$ for all $i \in \{0,\ldots, n\}$ \cite[Proposition 5.5.7]{GeckPfeiffer}. 
Notice that the boundary cases of $i=0$ and $i=n$ are the trivial and sign representations, respectively.
Because the irreducible $\mathfrak{B}_n$ characters may be afforded by $\mathbb{Q}[\mathfrak{B}_n]$-modules, every irreducible $\mathfrak{B}_n$-module is isomorphic to its dual \cite[Theorem 5.5.6]{GeckPfeiffer}.

There is also a type $B$ Frobenius characteristic map $F^B$, which sends the character $\chi^{(\lambda, \mu)}$ to the symmetric function $s_\lambda(\x)s_\mu(\y)$, where $\x$ and $\y$ are two distinct infinite sets of auxiliary variables.
We may denote $s_{\lambda}(\x) s_{\mu}(\y)$ by $s_{(\lambda,\mu)}(\x,\y)$. 
For a bitableau $Q$, we may denote $s_{\shape(Q^+)}(\x)s_{\shape(Q^-)}(\y)$ by $s_{\shape(Q)} (\x,\y)$, where $\shape(Q) = (\shape(Q^+), \shape(Q^-))$.

Regarding $R_{\mathfrak{B}_n}^{(k,j)}$, we can apply the type $B$ Frobenius characteristic map $F^B$ to the character of each multigraded component, and obtain the \emph{(type $B$) multigraded Frobenius series:} 
\begin{equation} \Frob(R_{\mathfrak{B}_n}^{(k,j)}; \mathbf{q}; \mathbf{u}) := \sum_{r_1,\ldots,r_k; s_1, \ldots, s_j \geq 0} F^B\ch\left((R_{\mathfrak{B}_n}^{(k,j)})_{r_1,\ldots,r_k; s_1, \ldots, s_j} \right)q_1^{r_1} \cdots q_k^{r_k}u_1^{s_1} \cdots u_j^{s_j}.\end{equation}
The context makes clear whether the type A or type B Frobenius characteristic map and Frobenius series is meant.

The Hall inner product on the ring of symmetric functions extends to
$\Lambda(\x) \otimes \Lambda(\y)$ by $\langle f_1(\x) g_1(\y) , f_2(\x) g_2(\y) \rangle = \langle f_1,f_2 \rangle \langle g_1,g_2\rangle$; thus $\{ s_\lambda(\x) s_\mu(\y) \}$ is an orthonormal basis and $F^B$ is an isometry.
We will use that $F^B(\mathbbm{1}_{\mathfrak{B}_n}) = h_n(\x)$ and $F^B(\varepsilon_{\mathfrak{B}_n}) = e_n(\y)$.

Let $K$ and $H$ be finite groups. 
For a $K$-module $P$ and $H$-module $Q$, by $P \boxtimes Q$ we mean the outer tensor product of $P$ and $Q$, which is a $K \times H$-module (see for example \cite[Section 1.2]{MR3839282}).
We will use that, for a $\mathfrak{B}_a$-module $U$ and a $\mathfrak{B}_b$-module $W$ with $a+b=n$, the map $F^B \ch$ takes $\Ind_{\mathfrak{B}_a \times \mathfrak{B}_b}^{\mathfrak{B}_n} (U \boxtimes W)$ to $F^B\ch(U) F^B\ch(W)$ (see \cite[Chapter I, Appendix B]{Macdonald}).

Chevalley \cite{Chevalley} (see also \cite{MR59914}) established that $R_{\mathfrak{B}_n}^{(1,0)}$ is isomorphic to the regular representation of $\mathfrak{B}_n$.
An important result of Stembridge gives the graded $\mathfrak{B}_n$-module structure of $R_{\mathfrak{B}_n}^{(1,0)}$ using the $\fmaj$ statistic.
\begin{theorem}[\cite{Stembridge}]\label{thm:Stembridge}
For $n \geq 1$,
\begin{equation}\label{eq:stembridge}
\Frob(R_{\mathfrak{B}_n}^{(1,0)};q) =  \sum_{Q \in \SYBT_n} q^{\fmaj(Q)} s_{\shape(Q)}(\x,\y).
\end{equation} 
\end{theorem}

The Hilbert series of $R_{\mathfrak{B}_n}^{(1,0)}$ is as follows.
\begin{theorem}[\cite{Chevalley, MR59914}]\label{thm:Chevalley-Shephard-Todd}
For $n \geq 1$,
    \begin{equation}
\Hilb(R_{\mathfrak{B}_n}^{(1,0)};q) =  [2n]_q!!.
\end{equation} 
\end{theorem}

In type $A$, from the (multigraded) Frobenius series of any $\mathfrak{S}_n$-module, one can recover the (multigraded) Hilbert series, by taking the inner product with $(h_{1})^n(\z)$. 
We now explain how this extends to the type $B$ setting. A type $B$ Frobenius series is of the form
\begin{equation}\label{eq:general-type-B-Frob}
    \sum_{(\lambda,\mu) \vdash n} c(\lambda,\mu) s_{\lambda}(\x) s_{\mu} (\y),
\end{equation}
for some coefficients $c(\lambda,\mu)$ which are polynomials in the grading variables with nonnegative integer coefficients.

Write $f^\lambda$ for the number of standard Young tableaux of shape $\lambda$.
Since $\dim V^{(\lambda,\mu)} = \binom{n}{|\lambda|} f^\lambda f^\mu$, the Hilbert series corresponding to equation~\eqref{eq:general-type-B-Frob} is
\begin{equation}
    \sum_{(\lambda,\mu) \vdash n} c(\lambda,\mu) \binom{n}{|\lambda|} f^\lambda f^\mu.
\end{equation}

Since $\langle s_\lambda , h_{(1^i)} \rangle = f^\lambda \delta_{|\lambda|,i}$, in order to obtain the Hilbert series, pair a type $B$ Frobenius series with
\begin{equation}\label{eq:to-pair}
    \sum_{i=0}^n \binom{n}{i} h_{(1^i)}(\x)  h_{(1^{n-i})}(\y) = \left(h_1(\x) + h_1(\y) \right)^n = \sum_{i=0}^n \binom{n}{i} h_{(1^{n-i})}(\x)  h_{(1^{i})}(\y).
\end{equation}

\section{The coinvariant ring \texorpdfstring{$R_{\mathfrak{B}_n}^{(0,1)}$}{RBn(0,1)}}\label{sec:01}

In this section, we derive a monomial basis, Frobenius series, Hilbert series, and the dimension of $R_{\mathfrak{B}_n}^{(0,1)}$.
While all of these can be deduced from known results, we give direct proofs as a warm-up.

The coinvariant ring $R_{\mathfrak{B}_n}^{(0,1)}$ is simply the exterior algebra on $\theta_1,\ldots,\theta_n$.
Explicitly in anticommuting variables,
\begin{equation}
    R_{\mathfrak{B}_n}^{(0,1)} = \C[\theta_1,\ldots,\theta_n]/{\langle\C[\theta_1,\ldots,\theta_n]_+^{\mathfrak{B}_n}\rangle} = \C[\theta_1,\ldots,\theta_n],
\end{equation}
since $\C[\theta_1,\ldots,\theta_n]_+^{\mathfrak{B}_n} = \{0\}$, by Proposition~\ref{prop:Weyl-B_n-bf}.
Therefore it is easy to determine a monomial basis.\footnote{The monomial basis is a special case of the type-independent basis given by Kim--Rhoades \cite{MR4381935} for $R_{W}^{(0,2)}$.} 
For $T = \{t_1 <\cdots <t_m\} \subseteq \{1,\ldots,n\}$, denote the ordered product $\theta_{t_1} \cdots \theta_{t_m}$ by $\theta_T$.

\begin{proposition}
    For $n \geq 1$, a monomial basis for $R_{\mathfrak{B}_n}^{(0,1)}$ is given by $\{\theta_T : T \subseteq \{1,\ldots,n\} \}$. 
\end{proposition}

Turning our attention to the module structure, we recall the following type $A$ result.
\begin{proposition}[\cite{HaglundSergel}]\label{prop:HS-one-set}
For $n \geq 1$,
        \begin{equation}
        \Frob(R_n^{(0,1)}; u) = \sum_{i=0}^{n-1} u^i s_{(n-i,1^i)}(\z).
    \end{equation}
\end{proposition}

We prove the following type $B$ analogue; it is an easy extension of \cite[Lemma 4.10]{HaglundSergel} to type $B$.\footnote{Proposition~\ref{prop:one-set} follows as a corollary from Theorem~\ref{thm:explicit-0-2} (set $v=0$, which forces $\lambda_2 = \ell = 0$ in that formula).}
\begin{proposition}\label{prop:one-set}
For $n \geq 1$,
    \begin{equation}
        \Frob(R_{\mathfrak{B}_n}^{(0,1)}; u) = \sum_{i=0}^{n} u^i s_{(n-i)}(\x)s_{(1^i)}(\y).
    \end{equation}
\end{proposition}

\begin{proof}
Let $(R_{\mathfrak{B}_n}^{(0,1)})_{i}$ denote the degree $i$ homogeneous part of $R_{\mathfrak{B}_n}^{(0,1)}$. 
Since the defining ideal is zero, $R_{\mathfrak{B}_n}^{(0,1)} = \C[\theta_1,\ldots,\theta_n]$ is the full exterior algebra, so $(R_{\mathfrak{B}_n}^{(0,1)})_{i} = \wedge^i (R_{\mathfrak{B}_n}^{(0,1)})_{1}$.
The degree-one component is the span of $\theta_1, \ldots, \theta_n$, on which $\mathfrak{B}_n$ acts by signed permutations; this is the defining representation $V^{((n-1),(1))}$. 
Hence $(R_{\mathfrak{B}_n}^{(0,1)})_{i} \cong \wedge^i V^{((n-1),(1))} \cong V^{((n-i),(1^i))}$.
Summing over $i$ and applying the Frobenius characteristic map completes the proof.
\end{proof}

From the Frobenius series, we can determine the Hilbert series. 
(The Hilbert series also follows from the monomial basis.)

\begin{proposition}\label{prop:01-B_n-Hilb}
For $n \geq 1$,
        \begin{equation}
        \Hilb(R_{\mathfrak{B}_n}^{(0,1)}; u) = \sum_{i=0}^{n} u^i\binom{n}{i} = (1+u)^n.
    \end{equation}
Consequently, $\dim R_{\mathfrak{B}_n}^{(0,1)} = 2^n$.
\end{proposition}

\begin{proof}
    Take the inner product of equation~\eqref{eq:to-pair} with the Frobenius series of Proposition~\ref{prop:one-set}, which yields
    \begin{equation}
        \sum_{i=0}^n u^i \binom{n}{i} \langle s_{(n-i)}(\x), h_{(1^{n-i})}(\x)\rangle \langle s_{(1^i)}(\y), h_{(1^i)}(\y)\rangle = \sum_{i=0}^n u^i \binom{n}{i}.
    \end{equation}
\end{proof}

\section{The coinvariant ring \texorpdfstring{$R_{\mathfrak{B}_n}^{(0,2)}$}{RBn(0,2)}}\label{sec:02}
Kim--Rhoades gave a formula for the module structure of $R_{\mathfrak{B}_n}^{(0,2)}$, established a combinatorial formula for the Hilbert series, determined the dimension, and found a monomial basis.
In this section, we extend their work by giving an explicit and manifestly positive Frobenius series of $R_{\mathfrak{B}_n}^{(0,2)}$, and consequently, for the $\GL_2 \times \mathfrak{B}_n$-module structure of $R_{\mathfrak{B}_n}^{(0,2)}$.
We give a new proof of their modified Motzkin path formula for the Hilbert series, and record an equivalent Hilbert series in terms of Kostka numbers.

For any irreducible complex reflection group $W$, Kim--Rhoades computed the $(i,j)$-graded piece of $R_W^{(0,2)}$ as follows.
\begin{theorem}[{\cite[Theorem 4.2]{MR4381935}}]\label{thm:KR}
Let $n \geq 1$.
Let $W$ be an irreducible complex reflection group acting on its reflection representation $V \cong \C^n$ and let $0 \leq i,j \leq n$.
If $i+j > n$, we have $(R_W^{(0,2)})_{i,j} = 0$. If $i+j \leq n$, inside the Grothendieck group of $W$ we have
\begin{equation}
    [(R_W^{(0,2)})_{i,j}] = [\wedge^i V] \cdot [\wedge^j V^*] - [\wedge^{i-1} V] \cdot [\wedge^{j-1} V^*].
\end{equation}
By convention, $\wedge^{-1} V = 0$.
\end{theorem}

In the case of the symmetric group $\mathfrak{S}_n$, they gave a formula for the bigraded Frobenius series of $R_n^{(0,2)}$, using the Kronecker product of symmetric functions. 
Given partitions $\lambda, \mu, \nu \vdash n$, the \emph{Kronecker coefficient} $g(\lambda, \mu, \nu)$ is defined as the multiplicity of the $\mathfrak{S}_n$ irreducible representation $V^\nu$ inside $V^\lambda \otimes V^\mu$. 
Then define the \emph{Kronecker product} $s_\lambda * s_\mu = \sum_{\nu \vdash n} g(\lambda, \mu, \nu) s_\nu$.

\begin{theorem}[\cite{MR4381935}]\label{thm:KR-A02}
For $n\geq 1$,
    \begin{equation}
        \Frob(R_n^{(0,2)};u,v) = \sum_{0 \leq i+j \leq n-1} u^i v^j \left(s_{(n-i,1^i)}*s_{(n-j,1^j)} - s_{(n-i+1,1^{i-1})}*s_{(n-j+1,1^{j-1})} \right)(\z),
    \end{equation}
    where the Schur function of a non-partition is $0$.
\end{theorem}

The only irreducible characters occurring with nonzero multiplicity are those indexed by double hooks.
The following result gives explicit formulas.

\begin{proposition}[\cite{MR4381935,lentfer2026signcharactertriagonalfermionic}]\label{prop:02-all-characters}
The multiplicity of the sign character is
\begin{equation}
\langle \Frob(R_n^{(0,2)};u,v),s_{(1^n)}(\z) \rangle = [n]_{u,v}.
\end{equation}
For $0 \leq i \leq n-2$, the multiplicities of the trivial and hook characters are
\begin{equation}
\langle \Frob(R_n^{(0,2)};u,v),s_{(n-i,1^i)}(\z) \rangle = [i+1]_{u,v} + uv[i]_{u,v}.
\end{equation}
Let $\lambda \vdash n$ be a double hook, i.e., 
\begin{equation}
    \lambda = (\lambda_1, \lambda_2, 2^\ell, 1^{n-2\ell-\lambda_1-\lambda_2})
\end{equation} 
where $\lambda_1 \geq \lambda_2 \geq 2$.
Let $m := n - 2\ell - \lambda_1 - \lambda_2 + 2$. Then if $\lambda_1 = \lambda_2$, we have
\begin{equation}
\langle \Frob(R_n^{(0,2)};u,v),s_{\lambda}(\z) \rangle = (uv)^{\ell+\lambda_2-1} (uv[m-2]_{u,v} + [m-1]_{u,v} + [m]_{u,v}).
\end{equation}
If $\lambda_1 > \lambda_2$, we have
\begin{equation}
\langle \Frob(R_n^{(0,2)};u,v),s_{\lambda}(\z) \rangle = (uv)^{\ell+\lambda_2-1} (uv[m-2]_{u,v} + (uv+1)[m-1]_{u,v} + [m]_{u,v}).
\end{equation}
\end{proposition}

Our goal is to get equally concrete results for the type $B$ coinvariant ring.
We want a description of $R_{\mathfrak{B}_n}^{(0,2)}$ which is manifestly positive in terms of its decomposition into irreducible representations. 
We will use the Mackey tensor product formula.
\begin{theorem}[{\cite{MR974302}; for this version, see \cite[(44.3)]{CurtisReiner}}]\label{thm:mackey-tensor}
Let $G$ be a finite group with subgroups $H$ and $K$, and let $T \subseteq G$ be a set of representatives for the double cosets $H \backslash G / K$, so that $G = \bigsqcup_{g \in T} HgK$. 
Let $L$ be a representation of $H$ and $N$ a
representation of $K$. 
For $g \in T$, let ${}^{g}N$ denote the representation of $gKg^{-1}$ with ${}^{g}N(x) = N(g^{-1} x g)$ for $x \in gKg^{-1}$. 
Then
\begin{equation}
    \Ind_H^G L \otimes \Ind_K^G N \cong \bigoplus_{g \in T}
    \Ind_{H \cap gKg^{-1}}^G \left( \Res_{H \cap gKg^{-1}}^H L \otimes
    \Res_{H \cap gKg^{-1}}^{gKg^{-1}} {}^{g}N \right).
\end{equation}
\end{theorem}

Define $\CT(n,i,j)$ to be the set of all $2 \times 2$ \emph{contingency tables with margins $(n-i,i)$ and $(n-j,j)$} \cite{Pearson}. 
These are nonnegative integer matrices 
\begin{equation}
    M = \begin{bmatrix}
    M_{11} & M_{12} \\
    M_{21} & M_{22}
\end{bmatrix}
\end{equation} which satisfy $M_{11} + M_{12} = n-i$, $M_{21} + M_{22} = i$, $M_{11} + M_{21} = n-j$, and $M_{12} + M_{22} = j$.
By convention, $\CT(n,i,j)=\varnothing$ if $i<0$ or $j<0$.

\begin{example}\label{ex:contingency-tables}
    The matrices in $\CT(3,2,1)$ are 
    \begin{equation}
        \begin{bmatrix}
            1 & 0 \\ 1 & 1
        \end{bmatrix}
        \quad \text{ and } \quad 
                \begin{bmatrix}
            0 & 1 \\ 2 & 0
        \end{bmatrix} .
    \end{equation}
    The matrix in $\CT(3,1,0)$ is 
        \begin{equation}
        \begin{bmatrix}
            2 & 0 \\ 1 & 0
        \end{bmatrix}.
    \end{equation}
\end{example}

The following is a joint result with C. Ryba.

\begin{proposition}\label{prop:typeB-wedges}
Let $0 \leq i,j \leq n$. Then
\begin{equation}\label{eq:mackey-module}
    \wedge^i V \otimes \wedge^j V^* \cong \bigoplus_{M \in \CT(n,i,j)} \Ind_{\mathfrak{B}_{M_{11}} \times \mathfrak{B}_{M_{12}} \times \mathfrak{B}_{M_{21}} \times \mathfrak{B}_{M_{22}}}^{\mathfrak{B}_n}\left(\mathbbm{1}_{\mathfrak{B}_{M_{11}}} \boxtimes \varepsilon_{\mathfrak{B}_{M_{12}}} \boxtimes \varepsilon_{\mathfrak{B}_{M_{21}}} \boxtimes \mathbbm{1}_{\mathfrak{B}_{M_{22}}}\right),
\end{equation}
and consequently the Frobenius image of $[\wedge^i V] \cdot [\wedge^j V^*]$ is
\begin{equation}\label{eq:mackey-two-wedges}
    \sum_{M \in \CT(n,i,j)} h_{M_{11}}(\x)h_{M_{22}}(\x) e_{M_{12}}(\y) e_{M_{21}}(\y).
\end{equation}
\end{proposition}

\begin{proof}
The defining representation $V \cong V^{((n-1),(1))}$ of $\mathfrak{B}_n$ may be realized as
\begin{equation}
    V \cong \Ind_{\mathfrak{B}_{n-1} \times \mathfrak{B}_1}^{\mathfrak{B}_n} ( \mathbbm{1}_{\mathfrak{B}_{n-1}} \boxtimes \varepsilon_{\mathfrak{B}_1}).
\end{equation}
Taking exterior powers gives, for $0 \leq i \leq n$,
\begin{equation}
    \wedge^i V \cong \Ind_{\mathfrak{B}_{n-i} \times \mathfrak{B}_i}^{\mathfrak{B}_n} ( \mathbbm{1}_{\mathfrak{B}_{n-i}} \boxtimes \varepsilon_{\mathfrak{B}_i}) \cong V^{((n-i),(1^i))}.
\end{equation}
Since $V \cong V^*$, we also have that 
\begin{equation}
    \wedge^j V^* \cong \Ind_{\mathfrak{B}_{n-j} \times \mathfrak{B}_j}^{\mathfrak{B}_n} ( \mathbbm{1}_{\mathfrak{B}_{n-j}} \boxtimes \varepsilon_{\mathfrak{B}_j}) \cong V^{((n-j),(1^j))}.
\end{equation}

Set $G  = \mathfrak{B}_n$, $H = \mathfrak{B}_{n-i}\times \mathfrak{B}_{i}$, and $K =  \mathfrak{B}_{n-j} \times \mathfrak{B}_{j}$.
Let
    \begin{equation}
        L = \mathbbm{1}_{\mathfrak{B}_{n-i}} \boxtimes \varepsilon_{\mathfrak{B}_i} \text{ and } N = \mathbbm{1}_{\mathfrak{B}_{n-j}} \boxtimes \varepsilon_{\mathfrak{B}_j},
    \end{equation}
so that $\wedge^i V \cong \Ind_H^G L$ and $\wedge^j V^* \cong \Ind_K^G N$.

We use the Mackey tensor product formula (Theorem~\ref{thm:mackey-tensor}).
Since $G  \cong C_2 \wr \mathfrak{S}_n$, $H \cong C_2 \wr \mathfrak{S}_{n-i}\times C_2 \wr\mathfrak{S}_{i}$, and $K \cong C_2 \wr  \mathfrak{S}_{n-j} \times C_2 \wr \mathfrak{S}_{j}$ each contain $C_2^n$, the projection $G \twoheadrightarrow \mathfrak{S}_n$ induces a bijection
\begin{equation}
    H\backslash G /K \longleftrightarrow (\mathfrak{S}_{n-i} \times \mathfrak{S}_i) \backslash \mathfrak{S}_n / (\mathfrak{S}_{n-j} \times \mathfrak{S}_j).
\end{equation}
The right hand side is indexed by the contingency tables $\CT(n,i,j)$:
explicitly, $g \in \mathfrak{S}_n$ is in the double coset indexed by
$M \in \CT(n,i,j)$ if and only if
\begin{equation}
    \left|\{1,\ldots, n-i\} \cap g(\{1,\ldots,n-j\})\right| = M_{11},
\end{equation}
\begin{equation}
    \left|\{1,\ldots, n-i\} \cap g(\{n-j+1,\ldots,n\})\right| = M_{12},
\end{equation}
\begin{equation}
    \left|\{n-i+1,\ldots,n\} \cap g(\{1,\ldots, n-j\})\right| = M_{21},
\end{equation}
\begin{equation}
    \left|\{n-i+1,\ldots,n\} \cap g(\{n-j+1,\ldots,n\})\right| = M_{22}.
\end{equation}
These cardinalities depend only on the double coset of $g$: left multiplication by $\mathfrak{S}_{n-i} \times \mathfrak{S}_i$ permutes each of the sets $\{1,\ldots,n-i\}$ and $\{n-i+1,\ldots,n\}$, while right multiplication by $\mathfrak{S}_{n-j} \times \mathfrak{S}_j$ permutes each of $\{1,\ldots,n-j\}$ and $\{n-j+1,\ldots,n\}$.
Conversely, given $M \in \CT(n,i,j)$, split $\{1,\ldots,n-j\}$ into sets of sizes $M_{11}$ and $M_{21}$ and $\{n-j+1,\ldots,n\}$ into sets of sizes $M_{12}$ and $M_{22}$, and let $g \in \mathfrak{S}_n$ carry the first and third into $\{1,\ldots,n-i\}$ and the second and fourth into $\{n-i+1,\ldots,n\}$; any two such choices differ by left multiplication by $\mathfrak{S}_{n-i} \times \mathfrak{S}_i$ and right multiplication by $\mathfrak{S}_{n-j} \times \mathfrak{S}_j$.
Hence the association is a bijection; choose for each $M$ a representative $g_M \in \mathfrak{S}_n$.

For each $M$, the corresponding intersection is the Young subgroup of $\mathfrak{B}_n$ given by
\begin{equation}
    H \cap g_M K g_M^{-1} = \mathfrak{B}_{M_{11}} \times \mathfrak{B}_{M_{12}} \times \mathfrak{B}_{M_{21}} \times \mathfrak{B}_{M_{22}},
\end{equation}
where the four factors act on index sets of sizes $M_{11}, M_{12}, M_{21}, M_{22}$.

Under the embedding $H \cap g_M K g_M^{-1} \hookrightarrow H$, the first two factors $\mathfrak{B}_{M_{11}} \times \mathfrak{B}_{M_{12}}$ land in the $\mathfrak{B}_{n-i}$ factor of $H$ (on which $L$ is the trivial representation $\mathbbm{1}_{\mathfrak{B}_{n-i}}$); the last two factors $\mathfrak{B}_{M_{21}} \times \mathfrak{B}_{M_{22}}$ land in the $\mathfrak{B}_{i}$ factor of $H$ (on which $L$ is the sign representation $\varepsilon_{\mathfrak{B}_{i}}$). 
Hence
\begin{equation}
    \Res_{H \cap g_M K g_M^{-1}}^H L \cong \mathbbm{1}_{\mathfrak{B}_{M_{11}}} \boxtimes \mathbbm{1}_{\mathfrak{B}_{M_{12}}} \boxtimes \varepsilon_{\mathfrak{B}_{M_{21}}} \boxtimes \varepsilon_{\mathfrak{B}_{M_{22}}}.
\end{equation}

Next consider the embedding $H \cap g_M K g_M^{-1} \hookrightarrow g_M K g_M^{-1}$. The factors $\mathfrak{B}_{M_{11}} \times \mathfrak{B}_{M_{21}}$ land in the $\mathfrak{B}_{n-j}$ factor of $K$ (on which $N$ is the trivial representation $\mathbbm{1}_{\mathfrak{B}_{n-j}}$); the factors $\mathfrak{B}_{M_{12}} \times \mathfrak{B}_{M_{22}}$ land in the $\mathfrak{B}_{j}$ factor of $K$ (on which $N$ is the sign representation $\varepsilon_{\mathfrak{B}_{j}}$). 
Hence
\begin{equation}
    \Res_{H \cap g_M K g_M^{-1}}^{g_M K g_M^{-1}} {}^{g_M}N \cong \mathbbm{1}_{\mathfrak{B}_{M_{11}}} \boxtimes \varepsilon_{\mathfrak{B}_{M_{12}}} \boxtimes \mathbbm{1}_{\mathfrak{B}_{M_{21}}} \boxtimes \varepsilon_{\mathfrak{B}_{M_{22}}}.
\end{equation}
Tensoring these restricted representations together and simplifying, we obtain
\begin{equation}
    \Res_{H \cap g_M K g_M^{-1}}^H L \otimes \Res_{H \cap g_M K g_M^{-1}}^{g_M K g_M^{-1}} {}^{g_M}N \cong \mathbbm{1}_{\mathfrak{B}_{M_{11}}} \boxtimes \varepsilon_{\mathfrak{B}_{M_{12}}} \boxtimes \varepsilon_{\mathfrak{B}_{M_{21}}} \boxtimes \mathbbm{1}_{\mathfrak{B}_{M_{22}}}.
\end{equation}

The Mackey tensor product formula gives us
\begin{equation}
    \wedge^i V \otimes \wedge^j V^* \cong \bigoplus_{M \in \CT(n,i,j)} \Ind_{\mathfrak{B}_{M_{11}} \times \mathfrak{B}_{M_{12}} \times \mathfrak{B}_{M_{21}} \times \mathfrak{B}_{M_{22}}}^{\mathfrak{B}_n} \left(\mathbbm{1}_{\mathfrak{B}_{M_{11}}} \boxtimes \varepsilon_{\mathfrak{B}_{M_{12}}} \boxtimes \varepsilon_{\mathfrak{B}_{M_{21}}} \boxtimes \mathbbm{1}_{\mathfrak{B}_{M_{22}}}\right).
\end{equation}
Upon applying the type $B$ Frobenius characteristic map, each matrix $M$ contributes
\begin{equation}
    h_{M_{11}}(\mathbf{x}) h_{M_{22}}(\mathbf{x}) e_{M_{12}}(\mathbf{y}) e_{M_{21}}(\mathbf{y}).
\end{equation}
Summing over all $M \in \CT(n,i,j)$ yields the claimed Frobenius image.
\end{proof}

\begin{corollary}\label{cor:wedges-mult-free}
    For $0 \leq i,j \leq n$, the $\mathfrak{B}_n$-module $\wedge^i V \otimes \wedge^j V^*$ is multiplicity-free, and the irreducibles occurring are indexed by bipartitions $(\lambda,\mu) \vdash n$ with $\lambda$ of at most two rows and $\mu$ of at most two columns.
\end{corollary}

\begin{proof}
    By equation~\eqref{eq:mackey-two-wedges} and the Pieri rule, $\langle h_a(\x)h_b(\x), s_\lambda(\x)\rangle$ vanishes unless $\ell(\lambda) \leq 2$ and $\langle e_c(\y)e_d(\y), s_\mu(\y)\rangle = \langle h_c (\y)h_d(\y), s_{\mu'}(\y)\rangle$ vanishes unless $\mu_1 \leq 2$; in either case the coefficient is $0$ or $1$. 
    Since $i+j$ and $|\lambda|$ determine $M_{11}$, at most one $M \in \CT(n,i,j)$ contributes to the coefficient of any given $s_\lambda(\x)s_\mu(\y)$.
\end{proof}

\begin{proposition}\label{prop:module-structure}
    The bigraded Frobenius character series of $R_{\mathfrak{B}_n}^{(0,2)}$ is
\begin{equation}
\begin{aligned}
\Frob(R_{\mathfrak{B}_n}^{(0,2)}; u, v) &= \sum_{0 \leq i+j \leq n} u^i v^j \Bigg( \sum_{M \in \CT(n,i,j)} h_{M_{11}}(\x)h_{M_{22}}(\x) e_{M_{12}}(\y) e_{M_{21}}(\y) \\
&\qquad- \sum_{M \in \CT(n,i-1,j-1)} h_{M_{11}}(\x)h_{M_{22}}(\x) e_{M_{12}}(\y) e_{M_{21}}(\y)\Bigg).
\end{aligned}
\end{equation}
\end{proposition}

\begin{proof}
    This follows from Theorem~\ref{thm:KR} and Proposition~\ref{prop:typeB-wedges}.
\end{proof}

We give an example at $n=3$.
\begin{example}\label{ex:n3-frobenius}
    Using the contingency tables from Example~\ref{ex:contingency-tables}, we compute that the coefficient of $ u^2 v $ in $\Frob(R_{\mathfrak{B}_3}^{(0,2)}; u, v)$ is
    \begin{equation}
\begin{aligned}
        &h_1(\x)h_1(\x)e_0(\y)e_1(\y) + h_0(\x)h_0(\x)e_1(\y)e_2(\y) - h_2(\x)h_0(\x)e_0(\y)e_1(\y)\\
        &= s_{(1,1)}(\x)s_{(1)}(\y) + s_{\varnothing}(\x)s_{(1,1,1)}(\y) + s_{\varnothing}(\x)s_{(2,1)}(\y).
    \end{aligned}
\end{equation}
    A similar calculation for all $0 \leq i+j \leq 3$ yields 
        \begin{equation}
\begin{aligned}
        \Frob&(R_{\mathfrak{B}_3}^{(0,2)}; u, v) =
        (u^3+u^2v+uv^2+v^3)s_{\varnothing}(\x) s_{(1, 1, 1)}(\y) 
        + (u^2v+uv^2)s_{\varnothing}(\x) s_{(2, 1)}(\y)\\ 
        &+ (u^2+uv+v^2)s_{(1)}(\x) s_{(1, 1)}(\y)
        + uvs_{(1)}(\x) s_{(2)}(\y) 
        + (u^2v+uv^2)s_{(1,1)}(\x) s_{(1)}(\y)\\ 
        &+ (u+v)s_{(2)}(\x) s_{(1)}(\y) + uvs_{(2,1)}(\x) s_{\varnothing}(\y) + s_{(3)}(\x) s_{\varnothing}(\y).
    \end{aligned}
\end{equation}
\end{example}

We are now able to give an explicit bigraded Frobenius series for $R_{\mathfrak{B}_n}^{(0,2)}$.
\begin{theorem}\label{thm:explicit-0-2}
Let $(\lambda,\mu) \vdash n$ be a bipartition such that $\lambda$ has at most two rows and $\mu$ has at most two columns.
Write $\lambda = (\lambda_1,\lambda_2)$ and $\mu = (2^\ell, 1^m)$, where $\ell$ is the number of rows of $\mu$ of length $2$ and $m = n - \lambda_1 - \lambda_2 - 2\ell$.
Then we have
\begin{equation}
\langle \Frob(R_{\mathfrak{B}_n}^{(0,2)};u,v),s_{\lambda}(\x) s_{\mu}(\y) \rangle = (uv)^{\ell+\lambda_2} [m+1]_{u,v} = s_{(n-\ell-\lambda_1, \ell+\lambda_2)}(u,v),
\end{equation}
and $\langle \Frob(R_{\mathfrak{B}_n}^{(0,2)};u,v),s_{\lambda}(\x) s_{\mu}(\y) \rangle  = 0$ for all other bipartitions.
Equivalently,
\begin{equation}
    \Frob(R_{\mathfrak{B}_n}^{(0,2)};u,v) = \sum_{\substack{(\lambda, \mu)\vdash n,\\ \lambda = (\lambda_1,\lambda_2),\\ \mu = (2^\ell, 1^{n-\lambda_1-\lambda_2-2\ell}) }} s_{(n-\ell-\lambda_1, \ell+\lambda_2)}(u,v)s_{\lambda}(\x) s_{\mu}(\y).
\end{equation}
\end{theorem}

\begin{proof}
Recall that the Kostka coefficients $K_{\alpha,\beta} = \langle h_\beta, s_\alpha \rangle$ count the number of semistandard Young tableaux of shape $\alpha$ and content $\beta$. 
In the case where both the shape and content (of the same size) are at most two rows, $K_{(\alpha_1,\alpha_2),(\beta_1, \beta_2)} = 1$ if $\alpha_2 \leq \beta_1 \leq \alpha_1$, and $0$ otherwise.

Let $i, j \geq 0$ with $i + j \leq n$ and consider $M \in \CT(n,i,j)$.
Observe that $M_{22} = i+j - n + M_{11}$; since $i + j \leq n$, this gives $M_{11} \geq i+j - n + M_{11}$.
We have that 
\begin{equation}
    \langle h_{M_{11}}(\x)h_{M_{22}}(\x), s_{(\lambda_1,\lambda_2)} (\x) \rangle = K_{(\lambda_1,\lambda_2),(M_{11}, i+j-n+M_{11})}.
\end{equation}
This Kostka coefficient is $1$ if the partitions are of the same size, that is, $\lambda_1 + \lambda_2 = M_{11} + (i+j-n+M_{11})$, and if $\lambda_2 \leq M_{11} \leq \lambda_1$. 
Otherwise, it is $0$. 
Furthermore, this implies that the sum $i+j$ is constant as $i+j = n-2M_{11} + \lambda_1 + \lambda_2$.

Observe that $M_{12} = n-i-M_{11}$ and $M_{21} = n-j-M_{11}$. We have that 
\begin{equation}
\begin{aligned}
    &\langle e_{M_{12}}(\y)e_{M_{21}}(\y), s_{(2^\ell, 1^{n-2\ell-\lambda_1-\lambda_2})} (\y) \rangle\\ 
    &= K_{(n-\ell-\lambda_1-\lambda_2, \ell),(\max\{n-i-M_{11}, n-j-M_{11}\},\min\{n-i-M_{11}, n-j-M_{11}\} )},
\end{aligned}    
\end{equation}
because the transpose of $(2^\ell, 1^{n-2\ell-\lambda_1-\lambda_2})$ is $(n-\ell-\lambda_1-\lambda_2, \ell)$. 
This Kostka coefficient is $1$ if $\ell \leq \max\{n-i-M_{11}, n-j-M_{11}\} \leq n - \ell - \lambda_1 - \lambda_2$. 
Otherwise, it is $0$. 
From this, for the Kostka coefficient to be $1$, we deduce that $n - M_{11} - \ell \geq \min\{i,j\} \geq \lambda_1 + \lambda_2 -M_{11} + \ell$.
Note that the upper bound $\min\{i,j\} \leq n - M_{11} - \ell$ is always satisfied; 
we have $\min\{i,j\} \leq (i+j)/2 = (n-2M_{11} + \lambda_1+ \lambda_2)/2 \leq n-M_{11}-\ell$ because $(n-2M_{11} + \lambda_1+ \lambda_2)/2 \leq n-M_{11}-\ell$ is equivalent to $n-2\ell - \lambda_1 -\lambda_2 \geq 0$, which is always true since $\mu$ is a partition.

Write $a_{i,j}$ for the multiplicity of $s_\lambda(\x)s_\mu(\y)$ in equation~\eqref{eq:mackey-two-wedges}.
The two conditions above say that $a_{i,j} = 1$ if
\begin{equation}
    M_{11} = \tfrac{1}{2}\big(n + \lambda_1 + \lambda_2 - i - j\big) \in \Z, \quad
    \lambda_2 \leq M_{11} \leq \lambda_1, \quad
    \min\{i,j\} \geq \lambda_1 + \lambda_2 - M_{11} + \ell,
\end{equation}
and $a_{i,j} = 0$ otherwise. 
Note that $a_{i,j} \in \{0,1\}$ (Corollary~\ref{cor:wedges-mult-free}) because $i+j$ determines $M_{11}$.
(The size condition for the second Kostka coefficient is the same as for the first, namely $i + j = n - 2M_{11} + \lambda_1 + \lambda_2$.)
By Proposition~\ref{prop:module-structure}, the multiplicity of $s_\lambda(\x)s_\mu(\y)$ in $\Frob(R_{\mathfrak{B}_n}^{(0,2)};u,v)$ equals $\sum_{i+j \leq n} (a_{i,j} - a_{i-1,j-1})\,u^i v^j$.
Replacing $(i,j)$ by $(i-1,j-1)$ replaces $M_{11}$ by $M_{11}+1$ and leaves the third condition unchanged, so $a_{i,j} - a_{i-1,j-1}$ vanishes unless $M_{11} = \lambda_1$, where it equals $1$, or $M_{11} = \lambda_2 - 1$, where it equals $-1$.
The latter forces $i + j = n + \lambda_1 - \lambda_2 + 2 > n$, which lies outside the range of Proposition~\ref{prop:module-structure}.
Hence only $M_{11} = \lambda_1$ contributes, giving
\begin{equation}
    \sum_{\substack{{i+j = n-\lambda_1 + \lambda_2,}\\{\min\{i,j\} \geq \ell + \lambda_2}}} u^iv^j = (uv)^{\ell+ \lambda_2} [n-\lambda_1-\lambda_2 -2\ell +1]_{u,v}.
\end{equation}
\end{proof}

We deduce the following multiplicity-free decomposition as a $\GL_2 \times \mathfrak{B}_n$-module of $R_{\mathfrak{B}_n}^{(0,2)}$. 
Here $\mathbb{S}^{(a,b)}(\C^2)$ denotes the irreducible polynomial $\GL_2$-representation with highest weight $(a,b)$, whose character is $s_{(a,b)}(u,v)$.
\begin{corollary}\label{cor:multiplicity-free}
    As a $\GL_2 \times \mathfrak{B}_n$-module, we have the decomposition
    \begin{equation}
        R_{\mathfrak{B}_n}^{(0,2)} \cong \bigoplus_{\substack{(\lambda, \mu) \vdash n,\\
        \lambda = (\lambda_1,\lambda_2),\\
        \mu = (2^\ell, 1^{n-\lambda_1-\lambda_2 - 2\ell})}}
        \mathbb{S}^{(n-\ell-\lambda_1, \ell+\lambda_2)}(\C^2) \otimes V^{(\lambda, \mu)}.
    \end{equation}
\end{corollary}

\begin{proof}
    By \cite[Theorem 1.1]{Lentfer-Supersymmetry} (with $(k,j) = (0,2)$ and $G = \mathfrak{B}_n$), there is an isomorphism of $\GL_2 \times \mathfrak{B}_n$-modules
    \begin{equation}
        R_{\mathfrak{B}_n}^{(0,2)} \cong \bigoplus_{\nu} \bigoplus_{(\lambda,\mu) \vdash n} \left(U^\nu_{0|2} \otimes V^{(\lambda, \mu)}\right)^{\oplus c_{\nu,(\lambda,\mu)}},
    \end{equation}
    where $U^\nu_{0|2} = \mathbb{S}^{\nu'}(\C^2)$ is the simple polynomial $\GL_2$-module with character $s_\nu(0/u,v) = s_{\nu'}(u,v)$, and the bigrading records the $\GL_2$-weight.
    Comparing bigraded multiplicities with Theorem~\ref{thm:explicit-0-2} identifies the coefficients, giving the claimed decomposition.
\end{proof}
Compare with skew Howe duality \cite{Howe}: while the present decomposition is multiplicity-free, distinct bipartitions may carry the same irreducible $\GL_2$-representation. 
For instance, in Example~\ref{ex:n3-frobenius} at $n=3$, both $(\varnothing,(2,1))$ and $((1,1),(1))$ have $\mathbb{S}^{(2,1)}(\C^2)$.

Returning to the current thread, we extract a concise formula for the multiplicity of the characters indexed by bipartitions $((n-i),(1^i))$, which can be thought of as a type $B$ analogue of the hook-shaped partitions for type $A$.
In type $A$ the corresponding multiplicities for the trivial, hook, and sign characters are not given by a single formula (Proposition~\ref{prop:02-all-characters}).

\begin{corollary}\label{cor:typeB-hooks}
    For $0 \leq i \leq n$, the bigraded multiplicity of $V^{((n-i),(1^i))}$ in $R_{\mathfrak{B}_n}^{(0,2)}$ is given uniformly by
    \begin{equation}
        \langle \Frob(R_{\mathfrak{B}_n}^{(0,2)};u,v),s_{(n-i)}(\x) s_{(1^i)}(\y) \rangle = [i+1]_{u,v}.
    \end{equation}
    The cases $i=0$ and $i=n$ are the trivial and sign characters.
\end{corollary}

We recall the following definition from Kim--Rhoades \cite{MR4381935}. Let $\Pi(n)_{\geq 0}$ denote the set of \emph{modified Motzkin paths}, which is the set of all $n$-step lattice paths $\pi = (\pi_1,\ldots,\pi_n)$ in $\mathbb{Z}^2$ that start at the origin $(0,0)$, never go below the $x$-axis, and whose steps $\pi_i$ are of the following four types:
\begin{itemize}
    \item $U$: an up-step $(1,1)$,
    \item $H_\theta$: a horizontal step $(1,0)$ of weight $\theta_i$,
    \item $H_\xi$: a horizontal step $(1,0)$ of weight $\xi_i$,
    \item $D$: a down-step $(1,-1)$ of weight $\theta_i\xi_i$.
\end{itemize}
Let the weight of a path $\pi$ be the product of all weights of the steps in the path (where we consider up-steps $U$ to have weight 1).

Kim--Rhoades showed the following in a type-independent manner; here we use the type $B$ specialization. Let $\deg_\theta(\pi)$ denote the $\theta$-degree of the path $\pi$, and let $\deg_\xi(\pi)$ denote the $\xi$-degree of the path $\pi$.

\begin{proposition}[{\cite[Corollary 5.4]{MR4381935}}]\label{prop:KR-Hilb}
    For $n \geq 1$,
    \begin{equation}
        \Hilb(R_{\mathfrak{B}_n}^{(0,2)} ; u,v) = \sum_{\pi \in \Pi(n)_{\geq 0}} u^{\deg_\theta(\pi)} v^{\deg_\xi(\pi)}.
    \end{equation}
\end{proposition}

We will give a new proof of this result, using our Frobenius series of $R_{\mathfrak{B}_n}^{(0,2)}$.

\begin{proof}
Start with the expression for $\Frob(R_{\mathfrak{B}_n}^{(0,2)};u,v)$ from Theorem~\ref{thm:explicit-0-2}. Then take the inner product with $\sum_{i=0}^n \binom{n}{i} h_{(1^i)}(\x)  h_{(1^{n-i})}(\y)$ (from equation~\eqref{eq:to-pair}) to obtain the Hilbert series:
\begin{equation}\label{eq:Hilbert-Kostka}
\begin{aligned}
    \Hilb&(R_{\mathfrak{B}_n}^{(0,2)};u,v)\\ 
    &= \left\langle \!\sum_{\substack{(\lambda, \mu)\vdash n,\\ \lambda = (\lambda_1,\lambda_2),\\ \mu = (2^\ell, 1^{n-\lambda_1-\lambda_2-2\ell}) }} \!s_{(n-\ell-\lambda_1, \ell+\lambda_2)}(u,v)s_{\lambda}(\x) s_{\mu}(\y), \sum_{i=0}^n \binom{n}{i} h_{(1^i)}(\x)  h_{(1^{n-i})}(\y) \right\rangle\\
    &= \sum_{\substack{(\lambda, \mu)\vdash n,\\ \lambda = (\lambda_1,\lambda_2),\ i =|\lambda|,\\ \mu = (2^\ell, 1^{n-\lambda_1-\lambda_2-2\ell})}} \!\binom{n}{i} K_{\lambda, (1^i)} K_{\mu, (1^{n-i})}s_{(n-\ell-\lambda_1,\ell+\lambda_2)}(u,v).
\end{aligned}
\end{equation}

Now let us interpret this sum as a weighted sum over certain paths. The Kostka number 
$K_{\lambda, (1^i)}$ counts the number of standard Young tableaux of shape $(\lambda_1,\lambda_2)$, where $i = \lambda_1 + \lambda_2$.

A \textit{Dyck path of semilength $r$} is a lattice path in $\Z^2$ from $(0,0)$ to $(r,r)$ consisting of North $N=(0,1)$ and East $E = (1,0)$ steps that never goes below the diagonal line $y=x$.
A \textit{Dyck path prefix} is any early truncation of a Dyck path (still starting from $(0,0)$). 
Then it is well known that standard Young tableaux of a two-row shape $(\lambda_1,\lambda_2)$ are in bijection with Dyck path prefixes from $(0,0)$ to $(\lambda_2,\lambda_1)$ as follows. 
Initialize the Dyck path prefix at $(0,0)$. For $j$ from $1$ to $i$: if $j$ is in the first row (the row of length $\lambda_1$), then append a North step to the Dyck path prefix, while if $j$ is in the second row (the row of length $\lambda_2$), then append an East step to the Dyck path prefix.

Now, the number of underlying paths with $\lambda_1$ up-steps $U$, $\lambda_2$ down-steps $D$ and $n-i$ unlabeled horizontal steps $H$ is $\binom{n}{i}K_{\lambda,(1^i)}$. 
Choose the positions of the $H$ steps in $\binom{n}{i}$ ways, and then a word in $\{U,D\}$ with $\lambda_1$ up-steps and $\lambda_2$ down-steps that never dips below the $x$-axis in $K_{\lambda,(1^i)}$ ways (by reading $U$ as a North step and $D$ as an East step, this is a Dyck path prefix ending at $(\lambda_2,\lambda_1)$).

We now assign the labels $H_\theta$ or $H_\xi$ to the horizontal steps and keep track of the weights.
The total degree of all weights on a given path is
$\#H_\theta + \#H_\xi + 2\#D = (n-i) + 2\lambda_2 = n - \lambda_1 + \lambda_2$.
Recall that $s_{(n-\ell-\lambda_1, \ell+\lambda_2)}(u,v) =(uv)^{\ell+\lambda_2}[n-2\ell-\lambda_1-\lambda_2+1]_{u,v}$ and  
$s_{(n-\ell-\lambda_1 - \lambda_2, \ell)}(u,v) =(uv)^{\ell}[n-2\ell-\lambda_1-\lambda_2+1]_{u,v}$ by equation~\eqref{eq:two-row-schur}.

Now we extract the weights from the paths.
There are $\lambda_2$ down-steps, each of weight $\theta_i\xi_i$, contributing $(uv)^{\lambda_2}$. 
By equation~\eqref{eq:two-row-schur},
\begin{equation}
    s_{(n-\ell-\lambda_1, \ell+\lambda_2)}(u,v) = (uv)^{\lambda_2} s_{(n-i-\ell, \ell)}(u,v),
\end{equation}
so the remaining factor $s_{(n-i-\ell,\ell)}(u,v)$ has multiplicity $K_{\mu,(1^{n-i})}$.
Summing over all $\ell \geq 0$ exactly gives the generating function $(u+v)^{n-i}$ for labeling each of the $n-i$ unlabeled $H$ steps as either $H_\theta$ or $H_\xi$.
By the binomial theorem, and expanding the binomial coefficients as a telescoping sum, we have
\begin{equation}
\begin{aligned}
    (u+v)^{n-i} &= \sum_{p=0}^{n-i} \binom{n-i}{p} u^p v^{n-i-p}\\
    &= \sum_{\ell \geq 0} \left( \binom{n-i}{\ell} - \binom{n-i}{\ell-1 }\right) \sum_{p = \ell}^{n-i-\ell} u^p v^{n-i-p}\\
    &= \sum_{\ell \geq 0} \left( \binom{n-i}{\ell} - \binom{n-i}{\ell-1 }\right) (uv)^\ell [n-i-2\ell+1]_{u,v}\\
    &= \sum_{\ell \geq 0} K_{(2^\ell, 1^{n-i-2\ell}),(1^{n-i})} s_{(n-i-\ell,\ell)}(u,v),
\end{aligned}
\end{equation}
because $K_{(2^\ell, 1^{n-i-2\ell}),(1^{n-i})} = \binom{n-i}{\ell} - \binom{n-i}{\ell-1 }$.
This is the weight from all $H$ steps for all paths obtained from a single underlying path with $\lambda_1$ $U$ steps and $\lambda_2$ $D$ steps.
There are $\binom{n}{i} K_{\lambda,(1^i)}$ such underlying paths, so summing over all $\lambda$ produces every modified Motzkin path $\pi$ in $\Pi(n)_{\geq 0}$ exactly once, weighted by $u^{\deg_\theta(\pi)}v^{\deg_\xi(\pi)}$.
This sum is exactly equation~\eqref{eq:Hilbert-Kostka}, establishing the claimed formula. Figure~\ref{fig:motzkin} illustrates the case $n = 2$.
\end{proof}

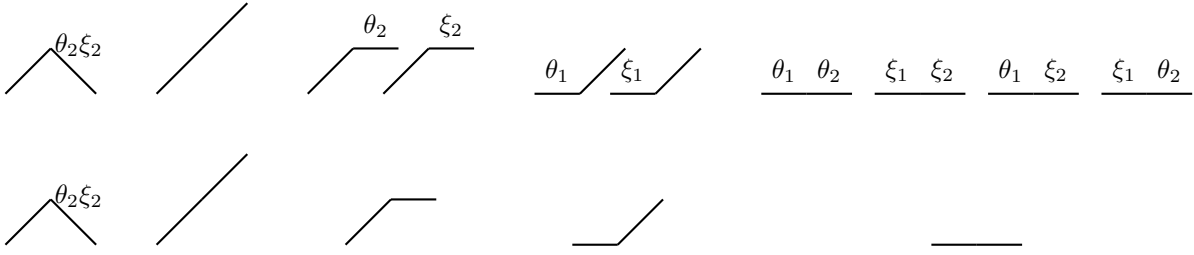
\begin{figure}
    \centering
\def\sep{2.9}
\begin{tikzpicture}
\node[anchor=south west] at (0,1) {\motzkin{up,down}{}};
\node[anchor=south west] at (2,1) {\motzkin{up,up}{}};
\node[anchor=south west] at (4,1) {\motzkin{up,flat-theta}{}};
\node[anchor=south west] at (5,1) {\motzkin{up,flat-xi}{}};
\node[anchor=south west] at (7,1) {\motzkin{flat-theta,up}{}};
\node[anchor=south west] at (8,1) {\motzkin{flat-xi,up}{}};
\node[anchor=south west] at (10,1) {\motzkin{flat-theta,flat-theta}{}};
\node[anchor=south west] at (11.5,1) {\motzkin{flat-xi,flat-xi}{}};
\node[anchor=south west] at (13,1) {\motzkin{flat-theta,flat-xi}{}};
\node[anchor=south west] at (14.5,1) {\motzkin{flat-xi,flat-theta}{}};
\node[anchor=south west] at (0,-1) {\motzkin{up,down}{}};
\node[anchor=south west] at (2,-1) {\motzkin{up,up}{}};
\node[anchor=south west] at (4.5,-1) {\motzkin{up,flat}{}};
\node[anchor=south west] at (7.5,-1) {\motzkin{flat,up}{}};
\node[anchor=south west] at (12.25,-1) {\motzkin{flat,flat}{}};
\end{tikzpicture}
    \caption{The modified Motzkin paths in $\Pi(2)_{\geq 0}$ (top), grouped by the underlying path in the alphabet $\{U,D,H\}$ from which they are obtained by labeling each $H$ step as $H_\theta$ or $H_\xi$ (bottom). 
    There are $\binom{n}{i}K_{\lambda,(1^i)}$ underlying paths with $\lambda_1$ up-steps and $\lambda_2$ down-steps, and each admits $2^{n-i}$ labelings.}
    \label{fig:motzkin}
\end{figure}

From the proof, we record a new formula for the Hilbert series.

\begin{corollary}\label{cor:new-Hilbert-B02}
For $n \geq 1$,
    \begin{equation}\begin{aligned}
    \Hilb(R_{\mathfrak{B}_n}^{(0,2)};u,v) &=\sum_{\substack{(\lambda, \mu)\vdash n,\\ \lambda = (\lambda_1,\lambda_2),\ i =|\lambda|,\\ \mu = (2^\ell, 1^{n-\lambda_1-\lambda_2-2\ell})}} \binom{n}{i} K_{\lambda, (1^i)} K_{\mu, (1^{n-i})}s_{(n-\ell-\lambda_1,\ell+\lambda_2)}(u,v).
\end{aligned}\end{equation}
\end{corollary}

\section{The sign character of \texorpdfstring{$R_{\mathfrak{B}_n}^{(0,3)}$}{RBn(0,3)}}\label{sec:sign}

The goal of this section is to prove the following result on the multiplicity of the sign character in $R_{\mathfrak{B}_n}^{(0,3)}$.
\begin{theorem}\label{thm:03sgn}
    For $n \geq 1$, the multiplicity of the sign character is given by
    \begin{equation}
        \langle \Frob(R_{\mathfrak{B}_n}^{(0,3)};u,v,w),s_{\varnothing}(\x) s_{(1^n)}(\y) \rangle =s_{(n)}(u,v,w).
    \end{equation}
\end{theorem}
Theorem~\ref{thm:03sgn} is a type $B$ analogue of the following result in type $A$.
\begin{theorem}[\cite{lentfer2026signcharactertriagonalfermionic}]\label{thm:A03-sign}
For $n \geq 1$,
\begin{equation}
    \begin{aligned}
        \langle \Frob(R_n^{(0,3)};u,v,w), s_{(1^n)}(\z)\rangle &= s_{(n-1)}(u,v,w)+s_{(n-2,1,1)}(u,v,w),
    \end{aligned}
\end{equation}
where a Schur function indexed by a non-partition is $0$.
\end{theorem}
Note that the type $B$ result has only one irreducible $\GL_3$-character, while the type $A$ result has two irreducible $\GL_3$-characters.
The proof has two halves: an upper bound and a lower bound.
For the upper bound, we observe that $R_{\mathfrak{B}_n}^{(0,3)}$ is a quotient of $R_{\mathfrak{B}_n}^{(0,2)} \otimes R_{\mathfrak{B}_n}^{(0,1)}$ and compute the sign multiplicity of the latter, using a type $B$ analogue of a result of Bessenrodt together with Corollary~\ref{cor:typeB-hooks}. 
For the lower bound, we exhibit $\theta_1\cdots\theta_n$ as a $\GL_3$ highest weight vector inside the alternating part of the harmonic space $H_{\mathfrak{B}_n}^{(0,3)}$, which is isomorphic to $R_{\mathfrak{B}_n}^{(0,3)}$.

Let $g^B((\nu^{(1)},\nu^{(2)}), (\lambda^{(1)},\lambda^{(2)}), (\mu^{(1)},\mu^{(2)}))$ denote the multiplicity of the irreducible $\mathfrak{B}_n$-representation indexed by $(\nu^{(1)},\nu^{(2)})$ in the tensor product of those indexed by $(\lambda^{(1)},\lambda^{(2)})$ and $(\mu^{(1)},\mu^{(2)})$. 
This is a \emph{type $B$ analogue of the Kronecker coefficient}.
(Certain values of $g^B$ have been previously determined; see for example \cite[Section 2]{Orellana}.)
We prove the following type $B$ analogue of a result of Bessenrodt \cite{Bessenrodt}. 

\begin{lemma}\label{lem:Bessenrodt-typeB}
Let $M$ and $N$ be finite-dimensional $\mathfrak{B}_n$-modules, multigraded by $\mathbf{q}_M$ and $\mathbf{q}_N$ respectively. 
Then
\begin{equation}\begin{aligned}
    \Frob&(M \otimes N; \mathbf{q}_M,\mathbf{q}_N)\\  &= \sum_{(\nu^{(1)},\nu^{(2)}) \vdash n} s_{(\nu^{(1)},\nu^{(2)})}(\x,\y) \sum_{(\lambda^{(1)},\lambda^{(2)}), (\mu^{(1)},\mu^{(2)}) \vdash n} \langle \Frob(M; \mathbf{q}_M), s_{(\lambda^{(1)},\lambda^{(2)})} (\x,\y)\rangle\\
    &\qquad \cdot\langle \Frob(N; \mathbf{q}_N), s_{(\mu^{(1)},\mu^{(2)})} (\x,\y) \rangle g^B((\nu^{(1)},\nu^{(2)}), (\lambda^{(1)},\lambda^{(2)}), (\mu^{(1)},\mu^{(2)})).
\end{aligned}\end{equation}
\end{lemma}

\begin{proof}
For a $\mathfrak{B}_n$-module $L$ and a bipartition $(\lambda^{(1)},\lambda^{(2)}) \vdash n$, the \emph{multiplicity space} of $V^{(\lambda^{(1)},\lambda^{(2)})}$ in $L$ is $\operatorname{Hom}_{\mathfrak{B}_n}(V^{(\lambda^{(1)},\lambda^{(2)})}, L)$. 
The evaluation map gives a natural isomorphism
\begin{equation}
L \cong \bigoplus_{(\lambda^{(1)},\lambda^{(2)}) \vdash n}
\operatorname{Hom}_{\mathfrak{B}_n}(V^{(\lambda^{(1)},\lambda^{(2)})}, L)
\otimes V^{(\lambda^{(1)},\lambda^{(2)})},
\end{equation}
where $\mathfrak{B}_n$ acts on the second tensor factor.

    Applying this multidegree by multidegree, we may decompose $M$ and $N$ into irreducibles:
    \begin{equation}
        M \cong \bigoplus_{(\lambda^{(1)},\lambda^{(2)}) \vdash n}
        M_{(\lambda^{(1)},\lambda^{(2)})} \otimes V^{(\lambda^{(1)},\lambda^{(2)})}
        \qquad \text{and} \qquad
        N \cong \bigoplus_{(\mu^{(1)},\mu^{(2)}) \vdash n}
        N_{(\mu^{(1)},\mu^{(2)})} \otimes V^{(\mu^{(1)},\mu^{(2)})},
    \end{equation}
    where the multiplicity spaces $M_{(\lambda^{(1)},\lambda^{(2)})}$ and $N_{(\mu^{(1)},\mu^{(2)})}$ are multigraded vector spaces with Hilbert series $\langle \Frob(M; \mathbf{q}_M), s_{(\lambda^{(1)},\lambda^{(2)})}(\x,\y) \rangle$ and $\langle \Frob(N; \mathbf{q}_N), s_{(\mu^{(1)},\mu^{(2)})}(\x,\y) \rangle$, respectively.
    Tensoring the two decompositions and applying the definition of $g^B$ to each $V^{(\lambda^{(1)},\lambda^{(2)})} \otimes V^{(\mu^{(1)},\mu^{(2)})}$ shows that the multiplicity space of $V^{(\nu^{(1)},\nu^{(2)})}$ in $M \otimes N$ is
    \begin{equation}
        \bigoplus_{(\lambda^{(1)},\lambda^{(2)}),\, (\mu^{(1)},\mu^{(2)}) \vdash n}
        \big(M_{(\lambda^{(1)},\lambda^{(2)})} \otimes
        N_{(\mu^{(1)},\mu^{(2)})}\big)^{\oplus\,
        g^B((\nu^{(1)},\nu^{(2)}),\, (\lambda^{(1)},\lambda^{(2)}),\,
        (\mu^{(1)},\mu^{(2)}))}.
    \end{equation}
    Taking Hilbert series of the multiplicity spaces gives the multigraded formula.
\end{proof}

The next result is the type $B$ analogue of a well known fact about Kronecker coefficients: $g(\lambda,\mu,(1^n)) = \delta_{\mu,\lambda'}$.
\begin{lemma}\label{lem:Kronecker-typeB}
For any bipartitions $(\lambda^{(1)},\lambda^{(2)})$ and $(\mu^{(1)},\mu^{(2)})$ of $n$,
    \begin{equation}
        g^B((\varnothing,(1^n)), (\lambda^{(1)},\lambda^{(2)}), (\mu^{(1)},\mu^{(2)})) = \begin{cases}
            1 \text { if } (\mu^{(1)},\mu^{(2)}) = (\lambda^{(2)'},\lambda^{(1)'}),\\
            0 \text { otherwise.}
        \end{cases}
    \end{equation}
\end{lemma}

\begin{proof}
    Recall that $V^{(\lambda^{(1)},\lambda^{(2)})} \otimes V^{(\varnothing,(1^n))} \cong V^{(\lambda^{(2)'},\lambda^{(1)'})}$ \cite[Theorem 5.5.6]{GeckPfeiffer}. Then we write that
    \begin{equation}\begin{aligned}
        \langle V^{(\varnothing,(1^n))} ,V^{(\lambda^{(1)},\lambda^{(2)})} \otimes V^{(\mu^{(1)},\mu^{(2)})} \rangle&= \langle {V^{(\lambda^{(1)},\lambda^{(2)})}}^* \otimes V^{(\varnothing,(1^n))} , V^{(\mu^{(1)},\mu^{(2)})} \rangle\\
        &= \langle {V^{(\lambda^{(1)},\lambda^{(2)})}} \otimes V^{(\varnothing,(1^n))} , V^{(\mu^{(1)},\mu^{(2)})} \rangle\\
        &= \langle {V^{(\lambda^{(2)'},\lambda^{(1)'})}}, V^{(\mu^{(1)},\mu^{(2)})} \rangle.
    \end{aligned}\end{equation}
    Here we used that every irreducible $\mathfrak{B}_n$-module $V$ is isomorphic to its dual $V^*$.
\end{proof}

We prove an upper bound on the multiplicity of the sign character in $R_{\mathfrak{B}_n}^{(0,3)}$.

\begin{proposition}\label{prop:upper-bound-03}
For $n \geq 1$ we have, coefficient-wise in the basis $\{s_\lambda(\x)s_\mu(\y)\}$ and in the monomials in $u,v,w$,    \begin{equation}
        \langle \Frob(R_{\mathfrak{B}_n}^{(0,3)};u,v,w),s_{\varnothing}(\x) s_{(1^n)}(\y) \rangle \leq s_{(n)}(u,v,w).
    \end{equation}
\end{proposition}

\begin{proof}
By Proposition~\ref{prop:Weyl-B_n-bf}, $I_{\mathfrak{B}_n}^{(0,2)} = \langle \sum_{i=1}^n \theta_i\xi_i \rangle$ and $I_{\mathfrak{B}_n}^{(0,1)} = \{0\}$, so 
\begin{equation}
    R_{\mathfrak{B}_n}^{(0,2)} \otimes R_{\mathfrak{B}_n}^{(0,1)} = \C[\bm{\theta}_n,\bm{\xi}_n,\bm{\rho}_n]/\left\langle \sum_{i=1}^n \theta_i\xi_i\right\rangle.
\end{equation}
Since $I_{\mathfrak{B}_n}^{(0,3)} \supseteq \langle \sum_{i=1}^n \theta_i\xi_i\rangle$, the quotient map $\C[\bm{\theta}_n,\bm{\xi}_n,\bm{\rho}_n] \twoheadrightarrow R_{\mathfrak{B}_n}^{(0,3)}$ factors through $R_{\mathfrak{B}_n}^{(0,2)} \otimes R_{\mathfrak{B}_n}^{(0,1)}$, giving a surjection of trigraded $\mathfrak{B}_n$-modules $R_{\mathfrak{B}_n}^{(0,2)} \otimes R_{\mathfrak{B}_n}^{(0,1)} \twoheadrightarrow R_{\mathfrak{B}_n}^{(0,3)}$.
The surjection gives
\begin{equation}
    \Frob(R_{\mathfrak{B}_n}^{(0,3)}; u,v,w) \leq \Frob(R_{\mathfrak{B}_n}^{(0,2)} \otimes R_{\mathfrak{B}_n}^{(0,1)}; u,v,w).
\end{equation}
By extracting the coefficient of the sign character, we obtain
\begin{equation}
    \langle \Frob(R_{\mathfrak{B}_n}^{(0,3)}; u,v,w), s_{\varnothing}(\x) s_{(1^n)}(\y) \rangle \leq \langle \Frob(R_{\mathfrak{B}_n}^{(0,2)} \otimes R_{\mathfrak{B}_n}^{(0,1)}; u,v,w), s_{\varnothing}(\x) s_{(1^n)}(\y)\rangle.
\end{equation}

By Lemma~\ref{lem:Bessenrodt-typeB} and Lemma~\ref{lem:Kronecker-typeB}, we have
\begin{equation}\begin{aligned}
    \langle &\Frob(R_{\mathfrak{B}_n}^{(0,2)} \otimes R_{\mathfrak{B}_n}^{(0,1)}; u,v,w), s_{\varnothing}(\x) s_{(1^n)}(\y)\rangle\\ 
    &= \sum_{(\lambda^{(1)},\lambda^{(2)}) \vdash n} \langle \Frob(R_{\mathfrak{B}_n}^{(0,2)}; u,v), s_{(\lambda^{(1)},\lambda^{(2)})} \rangle \langle \Frob(R_{\mathfrak{B}_n}^{(0,1)}; w), s_{(\lambda^{(2)'},\lambda^{(1)'})} \rangle\\
    &= \sum_{i=0}^n [n-i+1]_{u,v}w^i\\
    &= s_{(n)}(u,v,w),
\end{aligned}\end{equation}
where we applied Corollary~\ref{cor:typeB-hooks} and Proposition~\ref{prop:one-set}.
\end{proof}

Next, we work towards the lower bound.
Recall $\det(g)$ denotes the determinant character of $g \in \mathfrak{B}_n$.
For any $\mathfrak{B}_n$-module $M$, define the \emph{alternating subspace} of $M$ by
\begin{equation}
    M^{\alt} = \{v \in M \, | \, g \cdot v = \det(g) v \text{ for all } g \in \mathfrak{B}_n \}.
\end{equation}

\begin{lemma}\label{lem:alt}
Let $M$ be a $\mathfrak{B}_n$-module, which is multigraded such that the grading is preserved under the action of $\mathfrak{B}_n$. 
Then
 \begin{equation}
     \langle \Frob(M; \mathbf{q}) , s_{\varnothing}(\x) s_{(1^n)}(\y) \rangle = \Hilb(M^{\alt};  \mathbf{q}).
 \end{equation}
\end{lemma}

\begin{proof}
    The antisymmetrizing operator
    \begin{equation}
        Y = \frac{1}{2^n n!}\sum_{g \in \mathfrak{B}_n} \det(g) g
    \end{equation}
    is an idempotent in the group algebra $\C[\mathfrak{B}_n]$, which projects $M$ onto $M^{\alt}$. 
    Since the sign character is 1-dimensional, the sign-isotypic component of $M$ has dimension equal to the multiplicity of the sign character in the Frobenius series. 
    Keeping track of the grading, which is preserved by the action of $\mathfrak{B}_n$, gives the result.
\end{proof}

We will establish the sign character of $R_{\mathfrak{B}_n}^{(0,3)}$, making use of the harmonic space $H_{\mathfrak{B}_n}^{(0,3)}$, discussed at length in Appendix~\ref{sec:harmonic}. 
We will use the following characterization of the harmonic space $H_{\mathfrak{B}_n}^{(0,3)}$, which is a special case of Corollary~\ref{cor:harmonics-explicit}.
\begin{proposition}\label{prop:B03-harmonics}
The \emph{triagonal fermionic type $B$ harmonics} are given by
\begin{equation}\label{eq:simplified-harmonics} 
H_{\mathfrak{B}_n}^{(0,3)} = \left\{ f \in \C[\bm{\theta}_n, \bm{\xi}_n, \bm{\rho}_n]\, \Bigg| \, \sum_{i=1}^n \partial_{\theta_i}\partial_{\xi_i} f = \sum_{i=1}^n \partial_{\theta_i}\partial_{\rho_i} f= \sum_{i=1}^n \partial_{\xi_i}\partial_{\rho_i} f =0\right\}.
\end{equation}
\end{proposition}
We will also use that $H_{\mathfrak{B}_n}^{(0,3)} \cong R_{\mathfrak{B}_n}^{(0,3)}$ as trigraded $\mathfrak{B}_n$-modules (by Theorem~\ref{thm:harmonics-iso}).

\begin{proof}[Proof of Theorem~\ref{thm:03sgn}]
    By Proposition~\ref{prop:upper-bound-03}, it only remains to establish the lower bound, which we do on the harmonics side.
    Consider the element $\theta_1\cdots \theta_n$, of which we record three properties.

    First, $\theta_1\cdots \theta_n$ is $\mathfrak{B}_n$-antisymmetric and nonzero: for $g = (\epsilon, \sigma) \in \mathfrak{B}_n$, where $\sigma \in \mathfrak{S}_n$ and $\epsilon = (\epsilon_1,\ldots,\epsilon_n) \in \{-1,+1\}^n$, we have $g \cdot \theta_i = \epsilon_i \theta_{\sigma (i)}$, so
    \begin{equation}
        g \cdot (\theta_1\cdots \theta_n) = \prod_{i=1}^n \epsilon_i \theta_{\sigma (i)} = \left(\prod_{i=1}^n \epsilon_i\right) \sgn(\sigma)\, \theta_1\cdots \theta_n = \det(g)\, \theta_1\cdots \theta_n,
    \end{equation}
    where we used that $\sgn(\sigma)\sigma(\theta_1\cdots \theta_{n}) = \theta_1\cdots \theta_{n}$ (see for example \cite[Chapter III, Section 7.3, Proposition 5]{BourbakiAlgebraI}) and $\det(g) = \sgn(\sigma)\prod_{i=1}^n \epsilon_i$.

    Second, since no variable from $\bm{\xi}_n$ or $\bm{\rho}_n$ occurs in $\theta_1\cdots \theta_n$, every operator that differentiates with respect to some $\xi_i$ or $\rho_i$ annihilates it.
    In particular, the three operators of Proposition~\ref{prop:B03-harmonics} annihilate it, so $\theta_1\cdots \theta_n \in H_{\mathfrak{B}_n}^{(0,3)}$.

    Third, $\theta_1\cdots \theta_n$ is a $\GL_3$ highest weight vector of weight $(n,0,0)$: the $\GL_3$-weight of a trihomogeneous vector is its tridegree in $(\bm{\theta}_n,\bm{\xi}_n,\bm{\rho}_n)$, and, with respect to the ordering $\bm{\theta}_n > \bm{\xi}_n > \bm{\rho}_n$, the raising operators $\sum_{i=1}^n \theta_i \partial_{\xi_i}$ and $\sum_{i=1}^n \xi_i \partial_{\rho_i}$ (instances of the polarization operators of equation~\eqref{eq:superderivations}) annihilate it.

By Lemmas~\ref{lem:polarization-commute} and~\ref{lem:polarization-stability}, $H_{\mathfrak{B}_n}^{(0,3)}$ is a $\GL_3 \times \mathfrak{B}_n$-module, and the polarization operators preserve the alternating subspace $(H_{\mathfrak{B}_n}^{(0,3)})^{\alt}$, which contains $\theta_1\cdots\theta_n$.
Since $\C[\bm{\theta}_n,\bm{\xi}_n,\bm{\rho}_n]$ is a semisimple $\GL_3$-representation, the $\GL_3$-submodule generated by $\theta_1\cdots\theta_n$ is the irreducible representation $\mathbb{S}^{(n)}(\C^3)$ of highest weight $(n,0,0)$; it is contained in $(H_{\mathfrak{B}_n}^{(0,3)})^{\alt}$, and its Hilbert series is the character $s_{(n)}(u,v,w)$.
By Theorem~\ref{thm:harmonics-iso} and Lemma~\ref{lem:alt}, we conclude that
    \begin{equation}
        \langle \Frob(R_{\mathfrak{B}_n}^{(0,3)}; u,v,w), s_{\varnothing}(\x) s_{(1^n)}(\y)\rangle = \Hilb((H_{\mathfrak{B}_n}^{(0,3)})^{\alt};u,v,w) \geq s_{(n)}(u,v,w).
    \end{equation}
    Combined with Proposition~\ref{prop:upper-bound-03}, this gives the claimed equality.
\end{proof}

\begin{remark}\label{rem:general-j-lower-bound}
    For all $n \geq 1$ and $j \geq 1$, the lower bound argument can easily be generalized to, coefficient-wise,
    \begin{equation}
        \langle \Frob(R_{\mathfrak{B}_n}^{(0,j)}; u_1,\ldots,u_j), s_{\varnothing}(\x) s_{(1^n)}(\y)\rangle \geq s_{(n)}(u_1,\ldots,u_j).
    \end{equation}
    However, for $j \geq 4$, additional terms start to appear, so the bound is not generally tight.
\end{remark}

\section{Three explicit characters}\label{sec:standard}

The \emph{standard character} of the symmetric group $\mathfrak{S}_n$ is the irreducible character indexed by the partition $(n-1,1)$. 
In this section we determine its multigraded multiplicity in $R_n^{(k,j)}$ for all $(k,j)$ (Theorem~\ref{thm:standard-character}) 
as well as the multigraded multiplicities of the $\mathfrak{B}_n$-characters $\chi^{((n-1),(1))}$ and $\chi^{((n-1,1), \varnothing)}$ in $R_{\mathfrak{B}_n}^{(k,j)}$ (Theorem~\ref{thm:standard-character-B_n}).
To our knowledge, these are the first characters, aside from the trivial character, to be determined in $R_n^{(k,j)}$ and $R_{\mathfrak{B}_n}^{(k,j)}$ for all $(k,j)$.

\subsection{Shared tools for types \texorpdfstring{$A$}{A} and \texorpdfstring{$B$}{B}} 
Define a \emph{super Schur function} (\cite{BereleRegev}; also see \cite[Appendix A.2.2]{ChengWangBook}) by 
\begin{equation}
    s_\lambda(\mathbf{q}/\mathbf{u}) = \sum_{\nu \subseteq \lambda} s_\nu(\mathbf{q}) s_{\lambda'/\nu'}(\mathbf{u}).
\end{equation}
The super Schur function $s_\lambda(\mathbf{q}/\mathbf{u})$ is homogeneous of degree $|\lambda|$, and vanishes in $k$ bosonic and $j$ fermionic variables precisely when $\lambda_{k+1} > j$ \cite{BereleRegev}.
In particular $s_\lambda(\mathbf{q}/\varnothing) = s_\lambda(\mathbf{q})$ and $s_\lambda(\varnothing/\mathbf{u}) = s_{\lambda'}(\mathbf{u})$.

We recall the following result from \cite{Lentfer-Supersymmetry}, which was conjectured in \cite[Conjecture 1]{Bergeron2020}.
Denote the set of partitions $\lambda$ with length $\ell(\lambda) \leq n$ which satisfy $\lambda_{k+1} \leq j$ by $P(k,j,n)$.

\begin{theorem}[{\cite[Corollary 1.3]{Lentfer-Supersymmetry}}]\label{thm:DSS}
    Fix a positive integer $n$. 
    For partitions $\lambda$ with $\ell(\lambda) \leq n$, and $\mu \vdash n$, there exist nonnegative integer coefficients $c_{\lambda,\mu}$ such that for any $(k,j)$, the multigraded Frobenius series of $R_n^{(k,j)}$ is
\begin{equation}\label{eq:DSS}
    \Frob(R_n^{(k,j)}; \mathbf{q};\mathbf{u}) = \sum_{\lambda \in P(k,j,n)} \sum_{\mu \vdash n} c_{\lambda,\mu} s_\lambda(\mathbf{q}/\mathbf{u})s_\mu(\z).
\end{equation}
\end{theorem}

The type $B$ version of the previous theorem is as follows. 

\begin{theorem}[{\cite[Theorem 1.2]{Lentfer-Supersymmetry}}]\label{thm:DSS-B_n}
    Fix a positive integer $n$. 
    For partitions $\lambda$ with $\ell(\lambda) \leq n$, and bipartitions $(\mu^{(1)},\mu^{(2)}) \vdash n$, there exist nonnegative integer coefficients $c_{\lambda,(\mu^{(1)},\mu^{(2)})}$ such that for any $(k,j)$, the multigraded Frobenius series of $R_{\mathfrak{B}_n}^{(k,j)}$ is
\begin{equation}\label{eq:DSS-B_n}
    \Frob(R_{\mathfrak{B}_n}^{(k,j)}; \mathbf{q};\mathbf{u}) = \sum_{\lambda \in P(k,j,n)} \sum_{(\mu^{(1)},\mu^{(2)}) \vdash n} c_{\lambda,(\mu^{(1)},\mu^{(2)})} s_\lambda(\mathbf{q}/\mathbf{u})s_{\mu^{(1)}}(\x)s_{\mu^{(2)}}(\y).
\end{equation}
\end{theorem}

The following is well known (see \cite[Proposition 1.2.2]{Haiman1994}); we include a short proof for completeness.

\begin{lemma}\label{lem:trivial}
    Fix $n \geq 1$ and let $G$ be $\mathfrak{S}_n$ or $\mathfrak{B}_n$. 
    For any $k,j$ nonnegative integers, the multiplicity of the trivial character of $G$ in $\Frob(R_G^{(k,j)}; \mathbf{q};\mathbf{u})$ is $1$; that is, $\langle \Frob(R_n^{(k,j)}; \mathbf{q};\mathbf{u}), s_{(n)}(\z)\rangle \allowbreak= 1$ and $\langle \Frob(R_{\mathfrak{B}_n}^{(k,j)}; \mathbf{q};\mathbf{u}), s_{(n)}(\x)s_\varnothing(\y)\rangle = 1$.
\end{lemma}

\begin{proof}
    By Remark~\ref{rem:averaging}, every invariant element of a multihomogeneous component of $R_G^{(k,j)}$ is the image of a $G$-invariant polynomial of that multidegree. 
    A $G$-invariant polynomial of nonzero multidegree lies in $S_+^{G} \subseteq I_G^{(k,j)}$, so its image is zero; the component of multidegree zero is spanned by $[1]$. Hence the trivial character occurs exactly once, in multidegree zero.
\end{proof}

We recall the following from \cite[Lemma 3.1]{Lentfer-Supersymmetry}.
\begin{lemma}\label{lem:one-set-subring}
    Let $k \geq 1$.
    The inclusion $\C[\bm{x}^{(1)}_n] \hookrightarrow \C[\bm{x}^{(1)}_n, \ldots, \bm{x}^{(k)}_n]$ induces injections $R_n^{(1,0)} \hookrightarrow R_n^{(k,0)}$ and $R_{\mathfrak{B}_n}^{(1,0)} \hookrightarrow R_{\mathfrak{B}_n}^{(k,0)}$.
    Equivalently,
    \begin{equation}
        I_n^{(k,0)} \cap \C[\bm{x}^{(1)}_n] = I_n^{(1,0)}
        \qquad \text{and} \qquad
        I_{\mathfrak{B}_n}^{(k,0)} \cap \C[\bm{x}^{(1)}_n]
        = I_{\mathfrak{B}_n}^{(1,0)}.
    \end{equation}
\end{lemma}

Now we are ready for the key computation: the $G_{n-1}$-invariants of $R_{G_n}^{(k,0)}$, for $G_n = \mathfrak{S}_n$ or $\mathfrak{B}_n$, treated uniformly.

\begin{proposition}\label{prop:parabolic-invariants}
    Let $(G_n, N)$ be $(\mathfrak{S}_n, n)$ or $(\mathfrak{B}_n, 2n)$, and let $G_{n-1} \leq G_n$ be the subgroup of elements that act trivially on the $n$th coordinate.
    For $k \geq 1$, the images in $R_{G_n}^{(k,0)}$ of the monomials
    \begin{equation}\label{eq:basis}
        \mathcal{B} = \left\{(x^{(1)}_n)^{i_1}\cdots(x^{(k)}_n)^{i_k} \;\middle|\;
        0 \leq i_1+\cdots+i_k \leq N-1\right\}
    \end{equation}
    form a basis of the invariant subspace $(R_{G_n}^{(k,0)})^{G_{n-1}}$.
    In particular, its multigraded Hilbert series is $1 + s_{(1)}(\mathbf{q}) + s_{(2)}(\mathbf{q}) + \cdots + s_{(N-1)}(\mathbf{q})$.
\end{proposition}

\begin{proof}
    Write $I = I_{G_n}^{(k,0)}$ and $R = R_{G_n}^{(k,0)}$, and let
    $[f]$ denote the image in $R$ of a polynomial $f$. 
    Let $\mathcal{V}$ be the set of exponent vectors $(r_1,\ldots,r_k)$ with $1 \leq r_1 + \cdots + r_k \leq N$, and with $r_1 + \cdots + r_k$ even when $G_n = \mathfrak{B}_n$. For $(r_1,\ldots,r_k) \in \mathcal{V}$ we have $p_{r_1,\ldots,r_k}(\bm{x}^{(1)}_n, \ldots, \bm{x}^{(k)}_n) \in I$ by Proposition~\ref{prop:Weyl-bf} (resp.\ Proposition~\ref{prop:Weyl-B_n-bf}) at $j = 0$, so splitting off the $n$th summand gives
    \begin{equation}\label{eq:pps-congruence}
        p_{r_1,\ldots,r_k}(\bm{x}^{(1)}_{n-1}, \ldots, \bm{x}^{(k)}_{n-1})
        \equiv -(x^{(1)}_n)^{r_1} \cdots (x^{(k)}_n)^{r_k} \pmod{I}.
    \end{equation}

    \emph{Spanning.} By Remark~\ref{rem:averaging}, every element of $(R)^{G_{n-1}}$ is the image of a $G_{n-1}$-invariant polynomial, and every $G_{n-1}$-invariant polynomial is a sum of products $g \cdot f$ with $g \in \C[\bm{x}^{(1)}_{n-1}, \ldots, \bm{x}^{(k)}_{n-1}]^{G_{n-1}}$ and $f \in \C[x^{(1)}_n, \ldots, x^{(k)}_n]$. 
    By Proposition~\ref{prop:Weyl-bf} (resp. Proposition~\ref{prop:Weyl-B_n-bf}) applied with $n-1$ in place of $n$ at $j = 0$, the invariant algebra $\C[\bm{x}^{(1)}_{n-1}, \ldots, \bm{x}^{(k)}_{n-1}]^{G_{n-1}}$ is generated by polarized power sums $p_{r_1,\ldots,r_k}(\bm{x}_{n-1}^{(1)},\ldots,\bm{x}_{n-1}^{(k)})$ with $1 \leq r_1+\cdots +r_k \leq n-1$ (resp. with $r_1+\cdots+r_k$ even and at most $2(n-1)$).
    These exponent vectors lie in $\mathcal{V}$, so each generator is congruent modulo $I$ to a monomial in $x^{(1)}_n, \ldots, x^{(k)}_n$ by equation~\eqref{eq:pps-congruence}.
    Hence $(R)^{G_{n-1}}$ is spanned by the images of the monomials of $\C[x^{(1)}_n, \ldots, x^{(k)}_n]$.

    Next we claim that $(x^{(1)}_n)^{N} = 0$ in $R_{G_n}^{(1,0)}$. 
    For $1 \leq i \leq n$, the polynomials $e_i(\bm{x}^{(1)}_n)$ (resp.\ $e_i\big((x^{(1)}_1)^2, \ldots, (x^{(1)}_n)^2\big)$) are $G_n$-invariant with zero constant term, hence lie in $I_{G_n}^{(1,0)}$. 
    Therefore, in $R_{G_n}^{(1,0)}[t]$,
    \begin{equation}
        \prod_{i=1}^{n}\big(t - x^{(1)}_i\big) = t^{n}, \qquad
        \text{resp.}\quad
        \prod_{i=1}^{n}\big(t^2 - (x^{(1)}_i)^2\big) = t^{2n},
    \end{equation}
    and substituting $t = x^{(1)}_n$ gives $(x^{(1)}_n)^{N} = 0$; that is, $(x^{(1)}_n)^{N} \in I_{G_n}^{(1,0)} \subseteq I$. 
    Since $I$ is stable under the polarization operators of equation~\eqref{eq:superderivations} (Lemma~\ref{lem:polarization-stability}), and $E_{k,1}^{i_k} \cdots E_{2,1}^{i_2}\,(x^{(1)}_n)^{N} = c\, (x^{(1)}_n)^{i_1}\cdots(x^{(k)}_n)^{i_k}$ for a nonzero scalar $c$ whenever $i_1 + \cdots + i_k = N$, every monomial of $\C[x^{(1)}_n, \ldots, x^{(k)}_n]$ of total degree $N$ (and hence of any higher degree) lies in $I$. 
    Hence, by the previous paragraph, $(R)^{G_{n-1}}$ is spanned by $\{[b] \mid b \in \mathcal{B}\}$.

    \emph{Linear independence.} Distinct elements of $\mathcal{B}$ have distinct multidegrees, and $I$ is multihomogeneous; therefore a nontrivial linear dependence among the $[b]$ would force a single monomial of $\mathcal{B}$ into $I$, say $(x^{(1)}_n)^{i_1}\cdots(x^{(k)}_n)^{i_k} \in I$ with $d := i_1+\cdots+i_k \leq N-1$. 
    Applying $E_{1,k}^{i_k} \cdots E_{1,2}^{i_2}$, which preserves $I$ by Lemma~\ref{lem:polarization-stability}, gives $c\,(x^{(1)}_n)^{d} \in I$ for a nonzero scalar $c$, so $(x^{(1)}_n)^{d} \in I \cap \C[\bm{x}^{(1)}_n] = I_{G_n}^{(1,0)}$ by Lemma~\ref{lem:one-set-subring}. 
    This is a contradiction: $\{x_1^{a_1}\cdots x_n^{a_n} : 0 \leq a_i \leq i-1\}$ is a basis of $R_n^{(1,0)}$ (the Artin basis \cite{Artin}), and likewise $\{x_1^{a_1}\cdots x_n^{a_n} : 0 \leq a_i \leq 2i-1\}$ is a basis of $R_{\mathfrak{B}_n}^{(1,0)}$ \cite[Section 5]{SwansonWallachII}; in either case $x_n^{d} \neq 0$ for $d \leq N-1$.

    Finally, $\mathcal{B}$ contains exactly one monomial of each multidegree $(i_1,\ldots,i_k)$ with $i_1+\cdots+i_k \leq N-1$, which gives the stated Hilbert series.
\end{proof}

\subsection{The standard character of \texorpdfstring{$R_{n}^{(k,j)}$}{Rn(k,j)}}\label{subsec:standard}

The following result was first announced in \cite{Lentfer-dissertation}.
\begin{theorem}\label{thm:standard-character} Fix $n \geq 2$. For any $k,j$
nonnegative integers,
    \begin{equation}
        \langle \Frob(R_n^{(k,j)}; \mathbf{q};\mathbf{u}),
        s_{(n-1,1)}(\z) \rangle
        = s_{(1)}(\mathbf{q}/\mathbf{u}) + s_{(2)}(\mathbf{q}/\mathbf{u})
        + \cdots + s_{(n-1)}(\mathbf{q}/\mathbf{u}).
    \end{equation}
\end{theorem}

\begin{proof}
    By the Pieri rule and Lemma~\ref{lem:trivial},
    \begin{equation}
        \langle \Frob(R_n^{(k,0)}; \mathbf{q}), s_{(n-1)}s_{(1)} (\z)\rangle
        = \langle \Frob(R_n^{(k,0)}; \mathbf{q}), s_{(n-1,1)} (\z)\rangle + 1.
    \end{equation}
    Since $s_{(n-1)}s_{(1)}(\z) = \Frob(\Ind_{\mathfrak{S}_{n-1}}^{\mathfrak{S}_n} \mathbbm{1}_{\mathfrak{S}_{n-1}})$, Frobenius reciprocity and Proposition~\ref{prop:parabolic-invariants} (with $N = n$) give
    \begin{equation}
        \langle \Frob(R_n^{(k,0)}; \mathbf{q}), s_{(n-1)}s_{(1)} (\z)\rangle
        = \Hilb\big((R_n^{(k,0)})^{\mathfrak{S}_{n-1}}; \mathbf{q}\big)
        = 1 + s_{(1)}(\mathbf{q}) + \cdots + s_{(n-1)}(\mathbf{q}).
    \end{equation}
    Hence $\langle \Frob(R_n^{(k,0)}; \mathbf{q}), s_{(n-1,1)} (\z)\rangle = s_{(1)}(\mathbf{q}) + \cdots + s_{(n-1)}(\mathbf{q})$ for every $k \geq 1$. 
    By Theorem~\ref{thm:DSS}, this multiplicity also equals $\sum_{\lambda} c_{\lambda,(n-1,1)}\, s_\lambda(\mathbf{q})$ with coefficients independent of $k$; taking $k \geq n$ and using the linear independence of the Schur polynomials $s_\lambda(q_1,\ldots,q_k)$ for $\ell(\lambda) \leq n \leq k$, we conclude that $c_{\lambda,(n-1,1)} = 1$ for $\lambda \in \{(1), (2), \ldots, (n-1)\}$ and $c_{\lambda,(n-1,1)} = 0$ otherwise. 
    Substituting these coefficients into equation~\eqref{eq:DSS} yields the claimed formula for all $(k,j)$.
\end{proof}

\subsection{Two characters of \texorpdfstring{$R_{\mathfrak{B}_n}^{(k,j)}$}{RBn(k,j)}}\label{subsec:standard-B_n}

We prove the following result, which constrains which universal series coefficients $c_{\lambda,(\mu^{(1)},\mu^{(2)})}$ from Theorem~\ref{thm:DSS-B_n} can be nonzero.

\begin{proposition}\label{prop:parity}
    We have $c_{\lambda,(\mu^{(1)},\mu^{(2)})} = 0$ unless $|\lambda| \equiv |\mu^{(2)}| \pmod{2}$.
\end{proposition}

\begin{proof}
Recall the construction of the irreducible $\mathfrak{B}_n$-characters (see \cite[Section 5.5]{GeckPfeiffer}): for a bipartition $(\mu^{(1)}, \mu^{(2)})$ of $n$ with $|\mu^{(1)}| = a$ and $|\mu^{(2)}| = b$,
\begin{equation}\label{eq:Bn-irrep-construction}
\chi^{(\mu^{(1)},\mu^{(2)})} =
\Ind_{\mathfrak{B}_a \times \mathfrak{B}_b}^{\mathfrak{B}_n}
\left( \tilde{\chi}^{\mu^{(1)}}
\boxtimes \big( \delta_b \otimes  \tilde{\chi}^{\mu^{(2)}}  \big) \right),
\end{equation}
where, for $\nu \vdash m$, $\tilde{\chi}^{\nu} := \chi^\nu \circ \pi_m$ denotes the pullback of the irreducible $\mathfrak{S}_m$-character $\chi^\nu$ along the canonical projection $\pi_m \colon \mathfrak{B}_m \twoheadrightarrow \mathfrak{S}_m$, which forgets the signs, and $\delta_m$ is the linear character of $\mathfrak{B}_m$ defined by $\delta_m(\epsilon, \sigma) = \epsilon_1 \cdots \epsilon_m$.

Consider the central element $-1 := ((-1,\ldots,-1), \mathrm{id}) \in \mathfrak{B}_n$.
By Schur's lemma, $-1$ acts on any module affording an irreducible character $\chi$ of $\mathfrak{B}_n$ by the scalar $\chi(-1)/\chi(1)$.
Write $\psi := \tilde{\chi}^{\mu^{(1)}} \boxtimes \big( \delta_b \otimes \tilde{\chi}^{\mu^{(2)}} \big)$ for the inducing character in equation~\eqref{eq:Bn-irrep-construction}.
Since $-1$ is central and lies in $\mathfrak{B}_a \times \mathfrak{B}_b$, the induced character formula gives $\chi^{(\mu^{(1)},\mu^{(2)})}(-1) = [\mathfrak{B}_n : \mathfrak{B}_a \times \mathfrak{B}_b]\, \psi(-1)$, and likewise at the identity, so $\chi^{(\mu^{(1)},\mu^{(2)})}(-1)/\chi^{(\mu^{(1)},\mu^{(2)})}(1) = \psi(-1)/\psi(1)$.
Since $\pi_a$ and $\pi_b$ kill the sign coordinates, $\tilde{\chi}^{\mu^{(1)}}$ and $\tilde{\chi}^{\mu^{(2)}}$ take the same value at $-1$ as at $1$, while $\delta_b\big((-1,\ldots,-1),\mathrm{id}\big) = (-1)^{b}$, so $\psi(-1) = (-1)^{b}\, \psi(1)$.
Hence $-1$ acts by the scalar $(-1)^{b} = (-1)^{|\mu^{(2)}|}$ in any module affording $\chi^{(\mu^{(1)},\mu^{(2)})}$.

On the other hand, $-1$ acts on every variable in $\C[\bm{x}^{(1)},\ldots,\bm{x}^{(k)},\bm{\theta}^{(1)},\ldots,\bm{\theta}^{(j)}]$ by $-1$, and hence acts on the component of $R_{\mathfrak{B}_n}^{(k,j)}$ of multidegree $(r_1,\ldots,r_k;s_1,\ldots,s_j)$ by $(-1)^{D}$, where $D = r_1 + \cdots + r_k + s_1 + \cdots + s_j$ is the total degree.
Comparing the two scalars, the multiplicity of $\chi^{(\mu^{(1)},\mu^{(2)})}$ in this component vanishes unless $D \equiv |\mu^{(2)}| \pmod{2}$.

Now fix $(\mu^{(1)},\mu^{(2)})$. In the coefficient $\sum_{\lambda} c_{\lambda,(\mu^{(1)},\mu^{(2)})}\, s_\lambda(\mathbf{q}/\mathbf{u})$ of $s_{\mu^{(1)}}(\x) s_{\mu^{(2)}}(\y)$ in equation~\eqref{eq:DSS-B_n}, every monomial of $s_\lambda(\mathbf{q}/\mathbf{u})$ has total degree $|\lambda|$, so by the previous paragraph the sum over those $\lambda$ with $|\lambda| \not\equiv |\mu^{(2)}| \pmod{2}$ vanishes identically for every $(k,j)$. 
Since the super Schur functions are linearly independent for $k, j$ sufficiently large \cite[Lemma 6.4]{BereleRegev} and the coefficients are independent of $(k,j)$, these $c_{\lambda,(\mu^{(1)},\mu^{(2)})}$ all vanish.
\end{proof}

We determine the multiplicities of $\chi^{((n-1),(1))}$ and $\chi^{((n-1,1),\varnothing)}$ in $\Frob(R_{\mathfrak{B}_n}^{(k,j)}; \mathbf{q};\mathbf{u})$.

\begin{theorem}\label{thm:standard-character-B_n} Fix $n \geq 2$. For any
$k,j$ nonnegative integers,
    \begin{equation}
        \langle \Frob(R_{\mathfrak{B}_n}^{(k,j)}; \mathbf{q};\mathbf{u}),
        s_{(n-1)}(\x)s_{(1)}(\y) \rangle
        = s_{(1)}(\mathbf{q}/\mathbf{u}) + s_{(3)}(\mathbf{q}/\mathbf{u})
        + \cdots + s_{(2n-1)}(\mathbf{q}/\mathbf{u}),
    \end{equation}
    and
        \begin{equation}
        \langle \Frob(R_{\mathfrak{B}_n}^{(k,j)}; \mathbf{q};\mathbf{u}),
        s_{(n-1,1)}(\x)s_{\varnothing}(\y) \rangle
        = s_{(2)}(\mathbf{q}/\mathbf{u}) + s_{(4)}(\mathbf{q}/\mathbf{u})
        + \cdots + s_{(2n-2)}(\mathbf{q}/\mathbf{u}).
    \end{equation}
\end{theorem}

\begin{proof}
    Since the regular representation $\reg_{\mathfrak{B}_1} = \mathbbm{1}_{\mathfrak{B}_1} \oplus
    \varepsilon_{\mathfrak{B}_1}$, induction gives
    \begin{equation}\label{eq:induction-B_n-1}
        \Frob(\Ind_{\mathfrak{B}_{n-1}}^{\mathfrak{B}_n}
        \mathbbm{1}_{\mathfrak{B}_{n-1}})
        = h_{n-1}(\x)\big(h_{1}(\x) + h_1(\y)\big)
        = s_{(n)}(\x) + s_{(n-1,1)}(\x) + s_{(n-1)}(\x)s_{(1)}(\y).
    \end{equation}
    By Frobenius reciprocity and Proposition~\ref{prop:parabolic-invariants} (with $N = 2n$), for every $k \geq 1$,
    \begin{equation}
    \begin{aligned}
        \langle \Frob(R_{\mathfrak{B}_n}^{(k,0)}; \mathbf{q}),\
        s_{(n)}(\x) + s_{(n-1,1)}(\x) + s_{(n-1)}(\x)s_{(1)}(\y) \rangle
        &= \Hilb\big((R_{\mathfrak{B}_n}^{(k,0)})^{\mathfrak{B}_{n-1}};
        \mathbf{q}\big) \\
        &= 1 + s_{(1)}(\mathbf{q}) + \cdots + s_{(2n-1)}(\mathbf{q}).
    \end{aligned}
    \end{equation}
    By Theorem~\ref{thm:DSS-B_n} and the linear independence of the Schur polynomials $s_\lambda(q_1,\ldots,q_k)$ for $\ell(\lambda) \leq n \leq k$, the type $B$ universal coefficients therefore satisfy, for every partition $\lambda$ with $\ell(\lambda) \leq n$,
    \begin{equation}\label{eq:combined-coefficients}
        \sum_{(\mu^{(1)},\mu^{(2)})}
        c_{\lambda,(\mu^{(1)},\mu^{(2)})} =
        \begin{cases}
            1 & \text{if } \lambda = \varnothing \text{ or }
            \lambda = (d) \text{ with } 1 \leq d \leq 2n-1,\\
            0 & \text{otherwise,}
        \end{cases}
    \end{equation}
    where the sum ranges over the bipartitions $((n),\varnothing)$, $((n-1,1),\varnothing)$, and $((n-1),(1))$ only.

    We now separate the three multiplicities.
    Since the coefficients $c_{\lambda,(\mu^{(1)},\mu^{(2)})}$ are nonnegative (Theorem~\ref{thm:DSS-B_n}), equation~\eqref{eq:combined-coefficients} shows that each $s_{(d)}(\mathbf{q}/\mathbf{u})$ belongs entirely to exactly one of the three bipartitions, and that no multi-row $\lambda$ occurs. 
    By Proposition~\ref{prop:parity}, the bipartition $((n-1),(1))$, with $|\mu^{(2)}| = 1$, can only yield $s_{(d)}(\mathbf{q}/\mathbf{u})$ with $d$ odd, while $((n),\varnothing)$ and $((n-1,1),\varnothing)$, with
    $|\mu^{(2)}| = 0$, can only yield those with $d$ even. 
    Hence $s_{(1)} + s_{(3)} + \cdots + s_{(2n-1)}$ is the coefficient of $s_{(n-1)}(\x)s_{(1)}(\y)$. 
    By Lemma~\ref{lem:trivial}, the coefficient of $s_{(n)}(\x)s_{\varnothing}(\y)$ is exactly $1$; the remaining $s_{(2)} + s_{(4)} + \cdots + s_{(2n-2)}$ is therefore the coefficient of $s_{(n-1,1)}(\x)s_{\varnothing}(\y)$, as claimed.
\end{proof}

\begin{remark}\label{rem:one-zero-specialization}
    At $(k,j) = (1,0)$, Theorems~\ref{thm:standard-character} and~\ref{thm:standard-character-B_n} specialize to
    \begin{equation}
        \langle \Frob(R_n^{(1,0)}; q), s_{(n-1,1)}(\z) \rangle = q + q^2 + \cdots + q^{n-1},
    \end{equation}
    \begin{equation}
        \langle \Frob(R_{\mathfrak{B}_n}^{(1,0)}; q), s_{(n-1)}(\x)s_{(1)}(\y) \rangle = q + q^3 + \cdots + q^{2n-1}, 
    \end{equation}
    and
    \begin{equation}
        \langle \Frob(R_{\mathfrak{B}_n}^{(1,0)}; q), s_{(n-1,1)}(\x)s_{\varnothing}(\y) \rangle = q^2 + q^4 + \cdots + q^{2n-2}.
    \end{equation}
    These can be verified directly from the classical formulas of Stanley \cite{MR526968} for $\Frob(R_n^{(1,0)}; q)$ and Stembridge for $\Frob(R_{\mathfrak{B}_n}^{(1,0)}; q)$ (Theorem~\ref{thm:Stembridge}).
    In the sum in equation~\eqref{eq:stembridge}, all the bitableaux $Q$ satisfy $\fmaj(Q) \equiv |\shape(Q^-)| \pmod{2}$, which is the $(k,j) = (1,0)$ instance of the parity constraint of Proposition~\ref{prop:parity}.
\end{remark}

\section*{Acknowledgements}
We would like to thank Sylvie Corteel, Mark Haiman, and Josh Swanson for helpful conversations, and especially Christopher Ryba, for introducing us to Mackey theory and for his help with the proof of Proposition~\ref{prop:typeB-wedges}.
J. L. was partially supported by the National Science Foundation Graduate Research Fellowship DGE-2146752. 

\appendix
\section{Harmonic Spaces}\label{sec:harmonic}

In this appendix we review diagonal harmonic spaces, establish several equivalent characterizations of them, and show that they are isomorphic to coinvariant rings. 
The content of this appendix is similar to the exposition of harmonic spaces in \cite[Section 1.3]{Haiman1994}, \cite{SwansonPreprint}, and \cite{SwansonWallach1}; we state similar (and well known) results adapted for the setting of this paper.

Recall that $S = \C[\bm{x}^{(1)}, \ldots, \bm{x}^{(k)}, \bm{\theta}^{(1)}, \ldots, \bm{\theta}^{(j)}]$, and let $G$ be $\mathfrak{S}_n$ or $\mathfrak{B}_n$, realized as the group of $n \times n$ permutation matrices or $n \times n$ signed permutation matrices, acting diagonally on $S$, with defining ideal $I_G^{(k,j)}$ and coinvariant ring $R_G^{(k,j)} = S/I_G^{(k,j)}$.
For a multidegree $\bm{d} = (r_1,\ldots,r_k;s_1,\ldots,s_j)$, let $S_{\bm{d}}$ denote the span of the monomials of degree $r_\ell$ in $\bm{x}^{(\ell)}$ and degree $s_m$ in $\bm{\theta}^{(m)}$; note that $S_{\bm{d}}$ is finite-dimensional.
Recall from Section~\ref{subsec:defining-ideals} that $I_G^{(k,j)}$ is multihomogeneous.

Haiman credits A. Garsia with observing the benefit of studying the diagonal coinvariant ring via the diagonal harmonics \cite{Haiman1994}, building on earlier connections between coinvariants and harmonics (see for example \cite{BergeronGarsia}).
Haiman \cite[Section 1.3]{Haiman1994} treats two sets of bosonic variables and any finite matrix group; Swanson \cite{SwansonPreprint} and Swanson--Wallach \cite{SwansonWallach1, SwansonWallachII} treat one set of bosonic and one set of fermionic variables, the former for unitary groups (which is no restriction for a finite group by \cite[Lemma 2.2]{SwansonPreprint}) and the latter for pseudo-reflection groups.
We record the setting we need here.
Of the hypotheses on the group, invariance of the apolar form is verified directly for signed permutation matrices in Lemma~\ref{lem:apolar-props}, and holds more generally whenever $G$ consists of unitary matrices \cite[Theorem 4.7]{SwansonPreprint}; the matrix entries being real-valued is used only in the implication (ii) $\Rightarrow$ (iii) of Theorem~\ref{thm:harmonics-pde}. 
In particular, no pseudo-reflection hypothesis is required (indeed, the diagonal action of $G$ on multiple sets of variables is not an action by pseudo-reflections).

\subsection{The apolar form}
We recall the definitions needed for the apolar form at the appropriate level of generality.
Differentiation with respect to a fermionic variable is defined algebraically, by
\begin{equation}\label{eq:interior-product}
    \partial_{\theta_i^{(\ell)}}\, \theta_{i_1}^{(\ell_1)} \cdots \theta_{i_m}^{(\ell_m)}
    = \begin{cases}
        (-1)^{c - 1}\, \theta_{i_1}^{(\ell_1)} \cdots \widehat{\theta_{i_c}^{(\ell_c)}} \cdots \theta_{i_m}^{(\ell_m)}
            & \text{ if $i = i_c$ and $\ell = \ell_c$ for some $c$,}\\
        0 & \text{ otherwise,}
    \end{cases}
\end{equation}
where $\widehat{\theta_{i}^{(\ell)}} $ denotes the omission of that variable.
Then extend $\C[\bm{x}^{(1)}, \ldots, \bm{x}^{(k)}]$-linearly.
The sign can be viewed equivalently as the parity of the number of adjacent transpositions needed to bring the variable $\theta_{i_c}^{(\ell_c)}$ in the ordered product to the leftmost position, so that it can then be contracted by the differential operator.
Each $\partial_{x_i^{(\ell)}}$ and $\partial_{\theta_i^{(m)}}$ satisfies the same commutation and anticommutation relations as the corresponding variable \cite[Lemma 4.2]{SwansonPreprint}.
Consequently, substituting $\partial_z$ for each variable $z$ in a polynomial $P \in S$ is an algebra homomorphism $P \mapsto P(\partial)$, whose image is supercommutative.
In particular, for the polarized power sums $p_{r_1,\ldots,r_k;s_1,\ldots,s_j}(\bm{x}^{(1)}_n, \ldots, \bm{x}^{(k)}_n, \bm{\theta}_n^{(1)}, \ldots, \bm{\theta}_n^{(j)})$, define 
\begin{equation}\label{eq:pps-operator}
        p_{\bm{r};\bm{s}}(\partial)
    := \sum_{i=1}^n \partial_{x_i^{(1)}}^{r_1} \cdots \partial_{x_i^{(k)}}^{r_k}\,
        \partial_{\theta_i^{(1)}}^{s_1} \cdots \partial_{\theta_i^{(j)}}^{s_j}.
\end{equation}
For a polynomial $P \in S$ whose monomials all have the same fermionic degree, denoted by $|P|$, set
\begin{equation}\label{eq:twisted-op}
    \partial_P := (-1)^{\binom{|P|}{2}}\, \overline{P}(\partial),
\end{equation}
where $\overline{P}$ denotes complex conjugation of coefficients. 
For general $P$, write $P = \sum_r P_r$ with $P_r$ of fermionic degree $r$ and set $\partial_P := \sum_r \partial_{P_r}$.
Note that $\partial_P$ and $P(\partial)$ differ by the sign in equation~\eqref{eq:twisted-op} and by the conjugation of coefficients: $P \mapsto P(\partial)$ is $\C$-linear, whereas $P \mapsto \partial_P$ is conjugate-linear, which is what makes the apolar form of Definition~\ref{def:apolar} sesquilinear.
The sign in~\eqref{eq:twisted-op} is needed because each monomial in $\partial_P$ is an ordered product of anticommutative operators.
For monomials $P,Q$ of fermionic degrees $p,q$ we have $\partial_{PQ} = (-1)^{\binom{p+q}{2}}\overline{P}(\partial)\overline{Q}(\partial)$ and $\partial_Q \partial_P = (-1)^{\binom{p}{2}+\binom{q}{2}}\overline{Q}(\partial)\overline{P}(\partial) = (-1)^{\binom{p}{2}+\binom{q}{2}+pq}\overline{P}(\partial)\overline{Q}(\partial)$ by supercommutativity; since $\binom{p+q}{2} = \binom{p}{2}+\binom{q}{2}+pq$, the two sides agree on monomials.
Thus for all $P,Q$ in $S$,
\begin{equation}\label{eq:antihom}
    \partial_{PQ} = \partial_Q\, \partial_P.
\end{equation}
\begin{definition}\label{def:apolar}
The \emph{apolar form} on $S$ is
\begin{equation}\label{eq:inner-product}
    \begin{aligned}
    \langle P, Q \rangle &:= \overline{\big(\partial_P\, Q\big)\big|_{\bm{x} = \bm{\theta} = 0}},
\end{aligned}
\end{equation}
that is, the complex conjugate of the constant term of $\partial_P Q$.
For $P$ of homogeneous fermionic degree, this equals
\begin{equation}
    (-1)^{\binom{|P|}{2}} P(\partial_{x_1^{(1)}},\ldots,
        \partial_{\theta_n^{(j)}}) \cdot \overline{Q}(x_1^{(1)}, \ldots,
        \theta_n^{(j)}) \big|_{\bm{x} = \bm{\theta} = 0}.
\end{equation}
\end{definition}
For $j = 0$, this is the classical extension of Rota's apolar form used in \cite[Section 1.3]{Haiman1994}.
For general $j$, the apolar form is an instance of the Hermitian form that Swanson--Wallach \cite[Section 2.4]{SwansonWallach1} associate to an arbitrary finite subgroup of a general linear group, with the same sesquilinearity convention, applied to the standard positive-definite Hermitian form on the span of the variables (for which signed permutation matrices are already unitary).
The form of Swanson \cite[Definitions 4.3 and 5.1]{SwansonPreprint} is instead defined without the sign of equation~\eqref{eq:twisted-op} and with the opposite sesquilinearity convention; it is therefore definite only up to the sign $(-1)^{\binom{r}{2}}$ on fermionic degree $r$ \cite[Lemma 5.2]{SwansonPreprint}.
The sign in equation~\eqref{eq:twisted-op} is exactly what reconciles these differences, and makes the present form positive definite.

\begin{lemma}\label{lem:apolar-props}
The apolar form~\eqref{eq:inner-product} is linear in $P$, conjugate-linear in $Q$, and satisfies $\langle Q, P\rangle = \overline{\langle P, Q\rangle}$. 
For $\bm{a}, \bm{b} \in \mathbb{Z}_{\geq 0}^{kn}$, and tuples $\bm{T}= (T_1,\ldots,T_j),\bm{U} = (U_1,\ldots,U_j)$ with $T_\ell, U_\ell \subseteq \{1,\ldots,n\}$ for each $\ell$,
\begin{equation}
    \langle \bm{x}^{\bm{a}}\, \bm{\theta}_{\bm{T}}, \bm{x}^{\bm{b}}\, \bm{\theta}_{\bm{U}} \rangle = \begin{cases}
        \bm{a}! & \text{ if } \bm{a} = \bm{b} \text{ and } \bm{T} = \bm{U},\\
        0 & \text{ otherwise,}
    \end{cases}
\end{equation}
where $\bm{x}^{\bm{a}} = \prod_{\ell,i}(x_i^{(\ell)})^{a_i^{(\ell)}}$, $\bm{a}! := \prod_{\ell, i}\, (a_i^{(\ell)})!$, and $\bm{\theta}_{\bm{T}} = \prod_{\ell=1}^j \theta^{(\ell)}_{T_\ell}$ in some fixed global order.
Thus distinct monomials are orthogonal (hence distinct multidegrees are orthogonal), and the form is a positive-definite Hermitian inner product on $S$ and on each $S_{\bm{d}}$. Moreover the form is $G$-invariant.
\end{lemma}

\begin{proof}
Sesquilinearity is immediate from equations~\eqref{eq:inner-product} and~\eqref{eq:twisted-op}. 
Let $\mu = \bm{x}^{\bm{a}}\bm{\theta}_{\bm{T}}$ and $\nu = \bm{x}^{\bm{b}}\bm{\theta}_{\bm{U}}$.
Suppose $\mu \neq \nu$. 
There are two cases.
  \begin{itemize}
   \item Suppose that $a_i^{(\ell)} > b_i^{(\ell)}$ for some $\ell, i$, or that $T_\ell \not\subseteq U_\ell$ for some $\ell$.
    In the first case, the factor $\partial_{x_i^{(\ell)}}^{a_i^{(\ell)}}$ of $\mu(\partial)$ annihilates $(x_i^{(\ell)})^{b_i^{(\ell)}}$, since the order of the derivative exceeds the exponent.
    In the second case, choose $t \in T_\ell \setminus U_\ell$; since partial derivatives never introduce variables, the argument of the factor $\partial_{\theta_t^{(\ell)}}$ of $\mu(\partial)$ is free of $\theta_t^{(\ell)}$ when it acts, and is annihilated.
    In either case $\mu(\partial) \nu = 0$.
   \item Otherwise $a_i^{(\ell)} \leq b_i^{(\ell)}$ for all $\ell, i$ and $T_\ell \subseteq U_\ell$ for all $\ell$, and
    \begin{equation}
        \mu(\partial) \nu = \pm \left( \prod_{\ell, i}
      \frac{b_i^{(\ell)}!}{\left( b_i^{(\ell)} - a_i^{(\ell)} \right)!} \right)
      \bm{x}^{\bm{b} - \bm{a}} \bm{\theta}_{\bm{U} \setminus \bm{T}},
    \end{equation}
    where $\bm{U} \setminus \bm{T} = (U_1 \setminus T_1, \dots, U_j \setminus T_j)$. 
    The assumption $\mu \neq \nu$ gives $(\bm{a}, \bm{T}) \neq (\bm{b}, \bm{U})$, so the monomial $\bm{x}^{\bm{b} - \bm{a}} \bm{\theta}_{\bm{U} \setminus \bm{T}}$ has positive degree.
  \end{itemize}
  In both cases $\mu(\partial) \nu$ has zero constant term, so $\langle \mu, \nu \rangle = 0$; in particular monomials of distinct multidegrees are orthogonal. 
For $\mu = \bm{x}^{\bm{a}} \theta_{t_1}^{(\ell_1)} \cdots \theta_{t_h}^{(\ell_h)}$, applying $\mu(\partial)$ to $\mu$ peels off the fermionic factors from the right (note that $\overline{\mu} = \mu$ as $\mu$ is a monomial), producing the sign $(-1)^{(h-1) + (h-2) + \cdots + 0} = (-1)^{\binom{h}{2}}$, which cancels the sign in equation~\eqref{eq:twisted-op}; the bosonic part contributes $\bm{a}!$. 
Hence $\langle \mu, \mu\rangle = \bm{a}!$, and $\langle P, P \rangle = \sum_\mu |c_\mu|^2\, \bm{a}_\mu! > 0$ for $P = \sum_\mu c_\mu \mu \neq 0$. 
Hermitian symmetry now follows by checking on monomials. 
For $G$-invariance: a signed permutation matrix $g$, acting diagonally, sends each monomial $\mu$ to $\pm \mu^g$ for a monomial $\mu^g$ of the same multidegree, and $\mu \mapsto \mu^g$ is a bijection preserving $\bm{a}_\mu!$; since $g$ has real entries it commutes with conjugation, and $\langle g P, g Q \rangle = \langle P, Q \rangle$ follows by expanding in monomials.
\end{proof}

\subsection{Diagonal harmonics}
We define the space of diagonal harmonics at the appropriate level of generality. 
Then we establish the isomorphism with coinvariant spaces, and give characterizations as the solution space to a system of partial differential equations.

\begin{definition}\label{def:harmonics}
The space of \emph{diagonal harmonics (in $k$ sets of bosonic and $j$ sets of fermionic variables)} is the orthogonal complement of the defining ideal with respect to the apolar form:
\begin{equation}
    H_G^{(k,j)} := \big(I_G^{(k,j)}\big)^{\perp} = \big\{ \eta \in S \ :\ \langle \omega, \eta \rangle = 0 \ \text{ for all } \omega \in I_G^{(k,j)} \big\}.
\end{equation}
\end{definition}

Since $I_G^{(k,j)}$ is multihomogeneous and distinct multidegrees are orthogonal, $H_G^{(k,j)}$ is multihomogeneous, and $\eta \in S$ is in  $H_G^{(k,j)}$ if and only if each multihomogeneous component of $\eta$ is. 
Since the form is $G$-invariant and $I_G^{(k,j)}$ is $G$-stable, each component $\big(H_G^{(k,j)}\big)_{\bm{d}}$ is a $G$-submodule of $S_{\bm{d}}$.
We are ready to establish the isomorphism between the harmonics and coinvariants; the proof is very similar to that of \cite[Proposition 1.3.1]{Haiman1994}.
\begin{theorem}\label{thm:harmonics-iso}
The projection $S \to R_G^{(k,j)}$ restricts to an isomorphism of multigraded $G$-modules
\begin{equation}
    H_G^{(k,j)} \xrightarrow{\ \sim\ } R_G^{(k,j)}, \quad \text{ given by } \quad \eta \mapsto \eta + I_G^{(k,j)}.
\end{equation}
\end{theorem}

\begin{proof}
Fix a multidegree $\bm{d}$ and abbreviate $I = I_G^{(k,j)}$ and $H = H_G^{(k,j)}$ for this proof. 
Since the apolar form is a positive-definite Hermitian inner product on the finite-dimensional space $S_{\bm{d}}$ (Lemma~\ref{lem:apolar-props}), we have the orthogonal decomposition $S_{\bm{d}} = I_{\bm{d}} \oplus H_{\bm{d}}$, and both summands are $G$-stable. 
Since $I$ is a $G$-stable ideal, the projection $\pi \colon S \to S/I$ is $G$-equivariant and preserves multidegrees. 
Its restriction to $H_{\bm{d}}$ is injective, since $\ker \pi \cap H_{\bm{d}} = I_{\bm{d}} \cap H_{\bm{d}} = 0$ by positive-definiteness, and surjective onto $(S/I)_{\bm{d}}$, since $\dim H_{\bm{d}} = \dim S_{\bm{d}} - \dim I_{\bm{d}} = \dim (S/I)_{\bm{d}}$. 
Summing over $\bm{d}$ gives the claim.
\end{proof}

\begin{remark}\label{rem:harmonics-lit}
    Theorem~\ref{thm:harmonics-iso} can also be deduced from \cite[Proposition 5.7]{SwansonPreprint}, whose hypothesis is that $G$ consists of unitary matrices, or by the argument of \cite[Lemma 5.4, equations (36), (37)]{SwansonWallach1}; the latter needs only that $G$ is a finite subgroup of a general linear group. 
    Likewise, the equivalence of (i) and (iv) in Theorem~\ref{thm:harmonics-pde} below can be proven as in \cite[Lemma 5.4, equation (35)]{SwansonWallach1}, since the pseudo-reflection hypothesis is not actually used in that argument.
\end{remark}

Since the isomorphisms above respect the multigrading, the multigraded Frobenius series may be computed on the harmonic side: $\Frob(R_n^{(k,j)}; \mathbf{q}; \mathbf{u})$ is obtained by applying $F\ch$ to the multigraded components of $H_n^{(k,j)}$, and likewise $\Frob(R_{\mathfrak{B}_n}^{(k,j)}; \mathbf{q}; \mathbf{u})$ from $F^{B}\ch$ applied to those of $H_{\mathfrak{B}_n}^{(k,j)}$.

We now give several descriptions of $H_G^{(k,j)}$.
Recall that $\partial_{\theta_i^{(t)}}^2 = 0$, so that only exponents $s_t \in \{0,1\}$ occur in equation~\eqref{eq:pps-operator}.

\begin{theorem}\label{thm:harmonics-pde}
For $\eta \in S$, the following are equivalent:
\begin{enumerate}
\renewcommand{\labelenumi}{(\roman{enumi})}
    \item $\eta \in H_G^{(k,j)}$;
    \item $\partial_\omega\, \eta = 0$ for all $\omega \in I_G^{(k,j)}$;
    \item $q(\partial)\, \eta = 0$ for all $q \in S_+^{G}$;
    \item $p(\partial)\, \eta = 0$ for every polarized power sum $p$ in the generating set of equation~\eqref{eq:gen-set-A}, for $G = \mathfrak{S}_n$, or of equation~\eqref{eq:gen-set-B}, for $G = \mathfrak{B}_n$.
\end{enumerate}
\end{theorem}

\begin{proof}
(ii) $\Rightarrow$ (i): By equation~\eqref{eq:inner-product}, $\langle \omega, \eta\rangle$ is the complex conjugate of the constant term of $\partial_\omega \eta$. 
Hence $\partial_\omega \eta = 0$ implies that $\langle \omega, \eta\rangle = 0$.

(i) $\Rightarrow$ (ii): Fix $\omega \in I_G^{(k,j)}$. For every monomial $\mu \in S$, we have $\omega \mu \in I_G^{(k,j)}$, so by (i) and equation~\eqref{eq:antihom},
\begin{equation}
    0 = \langle \omega\mu,\, \eta \rangle
      = \overline{\big(\partial_\mu\, (\partial_\omega\, \eta)\big)\big|_{\bm{x} = \bm{\theta} = 0}}.
\end{equation}
As $\mu$ ranges over all monomials, the constants $\big(\partial_\mu F\big)\big|_{\bm{x}=\bm{\theta}=0}$ are nonzero scalar multiples of the coefficients of $F$; applying this to $F = \partial_\omega \eta$ gives $\partial_\omega \eta = 0$.

(ii) $\Rightarrow$ (iii): Let $q \in S_+^{G}$. 
Since the action of $G$ preserves each multihomogeneous component of $S$, the space $S_+^{G}$ is multihomogeneous; write $q = \sum_{\bm{d}} q_{\bm{d}}$ with each multihomogeneous component $q_{\bm{d}} \in S_+^{G}$. 
In particular, each $q_{\bm{d}}$ has homogeneous fermionic degree.
Since $G$ consists of real-valued matrices, $S_+^{G}$ is stable under complex conjugation of coefficients, so $\overline{q_{\bm{d}}} \in S_+^{G} \subseteq I_G^{(k,j)}$ for every $\bm{d}$. 
Since $\overline{q_{\bm{d}}}$ has the same fermionic degree as $q_{\bm{d}}$ and $\overline{\overline{q_{\bm{d}}}} = q_{\bm{d}}$, equation~\eqref{eq:twisted-op} gives
\begin{equation}
    \partial_{\overline{q_{\bm{d}}}}
    = (-1)^{\binom{|q_{\bm{d}}|}{2}}\, q_{\bm{d}}(\partial).
\end{equation}
Applying (ii) with $\omega = \overline{q_{\bm{d}}}$ yields $q_{\bm{d}}(\partial)\, \eta = 0$ for every $\bm{d}$, and summing over $\bm{d}$ gives $q(\partial)\, \eta = 0$.

(iii) $\Rightarrow$ (iv): The polarized power sums in question lie in $S_+^{G}$.

(iv) $\Rightarrow$ (ii): Let $\omega \in I_G^{(k,j)}$. 
By Proposition~\ref{prop:Weyl-bf}, for $G = \mathfrak{S}_n$, or by Proposition~\ref{prop:Weyl-B_n-bf}, for $G = \mathfrak{B}_n$, we may write $\omega = \sum_i f_i p_i$ with each $p_i$ in the relevant generating set and $f_i \in S$, which we may take to be multihomogeneous. 
By equations~\eqref{eq:antihom} and~\eqref{eq:twisted-op}, the real coefficients of the $p_i$, and the supercommutativity of the image of $P \mapsto P(\partial)$, we write
\begin{equation}
    \partial_{f_i p_i} = \partial_{p_i}\, \partial_{f_i}
        = \pm\, p_i(\partial)\, \overline{f_i}(\partial)
        = \pm\, \overline{f_i}(\partial)\, p_i(\partial),
\end{equation}
so $\partial_{f_i p_i}\, \eta = \pm\, \overline{f_i}(\partial)\big(p_i(\partial)\, \eta\big) = 0$ by (iv). 
Summing over $i$ gives $\partial_\omega \eta = 0$.
\end{proof}

For clarity, we explicitly record the description of harmonics as the solution space to a system of partial differential equations. 
\begin{corollary}\label{cor:harmonics-explicit}
The harmonic spaces may be expressed as solutions to a system of partial differential equations:
\begin{align}
    H_n^{(k,j)}
        &= \left\{ \eta \in S \ \middle|\ 
        \begin{array}{l}
              p_{\bm{r};\bm{s}}(\partial)\,\eta = 0, \\
              \text{for all } 1 \leq \textstyle\sum_t r_t + \sum_t s_t \leq n,\\ 
              r_1,\ldots,r_k \in \Z_{\geq 0}, s_1,\ldots,s_j \in \{0,1\} 
            \label{eq:harm-A}
        \end{array}
        \right\},\\
    H_{\mathfrak{B}_n}^{(k,j)}
        &= \left\{ \eta \in S \ \middle|\ 
        \begin{array}{l}
             p_{\bm{r};\bm{s}}(\partial)\,\eta = 0,  \\
             \text{for all } 2 \leq \textstyle\sum_t r_t + \sum_t s_t \leq 2n \text{ even},\\ 
             r_1,\ldots,r_k \in \Z_{\geq 0}, s_1,\ldots,s_j \in \{0,1\}
        \end{array}
              \right\}.
            \label{eq:harm-B}
\end{align}
\end{corollary}

\subsection{Polarization operators}\label{subsec:polarization}

The general linear Lie superalgebra $\mathfrak{gl}(k|j)$ acts on $S$ by superderivations. 
Recall from \cite[Section 1.1]{ChengWangBook} that a linear map $D \colon S \to S$ is a \emph{superderivation of parity $|D| \in \Z_2$} if $D$ shifts the $\Z_2$-grading of $S$ by $|D|$ and satisfies $D(fh) = D(f)\,h + (-1)^{|D|\,|f|} f\, D(h)$ for homogeneous $f, h \in S$, where the parity $|f| \in \Z_2$ of $f$ is its fermionic degree, as in equation~\eqref{eq:twisted-op}, reduced mod $2$.
An even superderivation ($|D| = 0$) is a derivation in the usual sense.
For a thorough introduction to Lie superalgebras, we refer to \cite{ChengWangBook}.
For $1 \leq a,b \leq k$ and $1 \leq c,d \leq j$, we have \emph{polarization operators}
\begin{equation}\label{eq:superderivations}
    E_{a,b} = \sum_{i=1}^n x_i^{(a)} \partial_{x_i^{(b)}},\quad
    E_{a,d'} = \sum_{i=1}^n x_i^{(a)} \partial_{\theta_i^{(d)}},\quad
    E_{c',b} = \sum_{i=1}^n \theta_i^{(c)} \partial_{x_i^{(b)}},\quad
    E_{c',d'} = \sum_{i=1}^n \theta_i^{(c)} \partial_{\theta_i^{(d)}},
\end{equation}
which realize the action of $\mathfrak{gl}(k|j)$ on $S$ \cite[Lemma 5.13]{ChengWangBook}. 

\begin{lemma}\label{lem:polarization-commute}
    The polarization operators from equation~\eqref{eq:superderivations} commute with the diagonal actions of $\mathfrak{S}_n$ and $\mathfrak{B}_n$.
\end{lemma}

\begin{proof}
Identify the span of the variables with $\C^n \otimes \C^{k|j}$, where $\C^{k|j} = \C^k \oplus \C^j$ with $\C^k$ even and $\C^j$ odd; then $G \subseteq \GL_n$ acts on the first tensor factor and the polarization operators act on the second, so the two actions commute on the variables.
    Since $g \in G$ acts on $S$ as an (even) algebra automorphism, $g E g^{-1}$ is again a superderivation of $S$; it agrees with $E$ on the variables, and a superderivation is determined by its values on the variables, so $g E g^{-1} = E$.
\end{proof}

\begin{lemma}\label{lem:polarization-stability}
    Let $G$ be $\mathfrak{S}_n$ or $\mathfrak{B}_n$ acting diagonally on $S$. Then:
    \begin{enumerate}
        \item[(a)] The defining ideal $I_G^{(k,j)}$ is stable under each polarization operator from equation~\eqref{eq:superderivations}.
        \item[(b)] The harmonic space $H_G^{(k,j)}$ is stable under each polarization operator, and hence is a $\mathfrak{gl}(k|j)$-submodule of $S$. 
        \item[(c)] $H_G^{(0,j)}$ is a $\GL_j$-submodule for the diagonal action of $\GL_j$ on $S = \C[\bm{\theta}^{(1)}, \ldots, \bm{\theta}^{(j)}]$.
    \end{enumerate}
\end{lemma}

\begin{proof}
    (a) This is established in the proof of \cite[Theorem 1.1]{Lentfer-Supersymmetry}, where it is shown that, for any finite $G \subset \GL_n$, the invariant space $S_+^G$ and hence the ideal $\langle S_+^G \rangle$ are $U(\mathfrak{gl}(k|j)) \otimes \C[G]$-submodules of $S$.

(b) By Corollary~\ref{cor:harmonics-explicit}, $\eta \in S$ is in $H_G^{(k,j)}$ if and only if $p(\partial)\eta = 0$ for every $p$ in the relevant generating set of equation~\eqref{eq:gen-set-A} or~\eqref{eq:gen-set-B}.
    We claim that for each polarization operator $E$ and each such $p$, the supercommutator $[p(\partial), E]$ is a scalar multiple of $p'(\partial)$ for some $p'$ (possibly $p$ itself) in the same generating set, or is zero.
    Indeed, $\ad(E)$ sends
    $\partial_{x_i^{(a)}} \mapsto -\partial_{x_i^{(b)}}$ for $E = E_{a,b}$,
    $\partial_{x_i^{(a)}} \mapsto -\partial_{\theta_i^{(d)}}$ for $E = E_{a,d'}$,
    $\partial_{\theta_i^{(c)}} \mapsto \partial_{x_i^{(b)}}$ for $E = E_{c',b}$, and
    $\partial_{\theta_i^{(c)}} \mapsto -\partial_{\theta_i^{(d)}}$ for $E = E_{c',d'}$,
    and annihilates the remaining derivatives. 
    Since $\ad(E)$ is a superderivation, applying it to each summand of $p(\partial)$ lowers one exponent by $1$ and raises another (possibly the same) by $1$, leaving the total degree (and hence the degree and the parity condition) unchanged (except it becomes $0$ if a fermionic exponent would exceed $1$).
    This proves the claim.
    Now for such an $\eta$ and any polarization operator $E$,
    \begin{equation}
        p(\partial) \left( E \eta \right)= (-1)^{|p(\partial)| \, |E|} E \big( p(\partial) \eta \big) + \big[ p(\partial), E \big] \eta = 0,
    \end{equation}
    since $p(\partial)\eta = 0$ and $[p(\partial), E]\,\eta$ is a scalar multiple of $p'(\partial)\eta = 0$.
    Hence $E\eta \in H_G^{(k,j)}$.

    (c) By (b), $H_G^{(0,j)}$ is a $\mathfrak{gl}(j)$-submodule of $S$; equivalently, it is a $\GL_j$-submodule.
\end{proof}

\section{Data for \texorpdfstring{$R_{\mathfrak{B}_n}^{(0,3)}$}{RBn(0,3)}}\label{sec:appendix}

By Theorem~\ref{thm:DSS-B_n}, we can write a type $B$ Frobenius series in relatively compact form. 
Any universal series coefficient $c_{\lambda,(\mu^{(1)},\mu^{(2)})}$ (of Theorem~\ref{thm:DSS-B_n}) determined here for a fixed $n$ will hold for all $(k,j)$.
All data were computed in SageMath and Macaulay2, by modifying \cite{Zabrocki-code} and  \cite{Bergeron-code}.

\subsection{Frobenius series of \texorpdfstring{$R_{\mathfrak{B}_n}^{(0,3)}$}{RBn(0,3)}}
We give examples of $\Frob(R_{\mathfrak{B}_n}^{(0,3)};u,v,w)$ for $1 \leq n \leq 4$. 
To read a Frobenius series here, the super Schur functions $s_\lambda(0/ u,v,w) = s_{\lambda'}(u,v,w)$ are written as $s_\lambda$ first. 
(By Theorem~\ref{thm:DSS-B_n}, $\lambda \in P(0,3,n)$, i.e., $\ell(\lambda) \leq n$, and $\lambda_1 \leq 3$.)
Subsequently, the type $B$ Frobenius character is written as $s_{(\mu^{(1)},\mu^{(2)})}(\x,\y)$.

\begin{equation}
\Frob(R_{\mathfrak{B}_1}^{(0,3)};u,v,w) = s_{(1)}s_{\varnothing,(1)}(\x,\y) + s_{\varnothing}s_{(1),\varnothing}(\x,\y)
\end{equation}

\begin{equation}
\begin{aligned}
\Frob(R_{\mathfrak{B}_2}^{(0,3)};u,v,w) &= s_{(1,1)}s_{\varnothing,(1,1)}(\x,\y)\\ 
&+ s_{(2)}s_{\varnothing,(2)}(\x,\y)\\ 
&+ (s_{(3)} + s_{(1)})s_{(1),(1)}(\x,\y)\\ 
&+ s_{(2)}s_{(1,1),\varnothing}(\x,\y)\\ 
&+ s_{\varnothing}s_{(2),\varnothing}(\x,\y)
\end{aligned}
\end{equation}

\begin{equation}
\begin{aligned}
\Frob(R_{\mathfrak{B}_3}^{(0,3)};u,v,w) &= s_{(1,1,1)}s_{\varnothing,(1,1,1)}(\x,\y)\\ 
&+ s_{(2,1)}s_{\varnothing,(2,1)}(\x,\y)\\ 
&+ s_{(3)}s_{\varnothing,(3)}(\x,\y)\\ 
&+ (s_{(3,1)} + s_{(1,1)})s_{(1),(1,1)}(\x,\y)\\ 
&+ (s_{(3,1)} + s_{(2)})s_{(1),(2)}(\x,\y)\\ 
&+ (s_{(3)} + s_{(2,1)})s_{(1,1),(1)}(\x,\y)\\ 
&+ s_{(3,1)}s_{(1,1,1),\varnothing}(\x,\y)\\ 
&+ (s_{(3)} + s_{(1)})s_{(2),(1)}(\x,\y)\\ 
&+ s_{(2)}s_{(2,1),\varnothing}(\x,\y)\\ 
&+ s_{\varnothing}s_{(3),\varnothing}(\x,\y)
\end{aligned}
\end{equation}

\begin{equation}
\begin{aligned}
\Frob(R_{\mathfrak{B}_4}^{(0,3)};u,v,w) &= 
s_{(1,1,1,1)}s_{\varnothing,(1,1,1,1)}(\x,\y)\\ 
&+ s_{(2,1,1)}s_{\varnothing,(2,1,1)}(\x,\y)\\ 
&+ s_{(2,2)}s_{\varnothing,(2,2)}(\x,\y)\\ 
&+ s_{(3,1)}s_{\varnothing,(3,1)}(\x,\y)\\ 
&+ s_{(3,3)}s_{\varnothing,(4)}(\x,\y)\\ 
&+ (s_{(3,1,1)} + s_{(1,1,1)})s_{(1),(1,1,1)}(\x,\y)\\ 
&+ (s_{(3,2)} + s_{(3,1,1)} + s_{(2,1)})s_{(1),(2,1)}(\x,\y)\\ 
&+ (s_{(3,2)} + s_{(3)})s_{(1),(3)}(\x,\y)\\
&+ (s_{(3,1)} + s_{(2,1,1)})s_{(1,1),(1,1)}(\x,\y)\\
&+ (s_{(3,3)} + s_{(3,1)} + s_{(2,2)})s_{(1,1),(2)}(\x,\y)\\
&+ (s_{(3,2)} + s_{(3,1,1)})s_{(1,1,1),(1)}(\x,\y)\\
&+ s_{(3,3)}s_{(1,1,1,1),\varnothing}(\x,\y)\\ 
&+ (s_{(3,3)} + s_{(3,1)} + s_{(1,1)})s_{(2),(1,1)}(\x,\y)\\ 
&+ (s_{(3,1)} + s_{(2)})s_{(2),(2)}(\x,\y)\\ 
&+ (s_{(3,2)} + s_{(3)} + s_{(2,1)})s_{(2,1),(1)}(\x,\y)\\
&+ s_{(3,1)}s_{(2,1,1),\varnothing}(\x,\y)\\
&+ s_{(2,2)}s_{(2,2),\varnothing}(\x,\y)\\ 
&+ (s_{(3)} + s_{(1)})s_{(3),(1)}(\x,\y)\\ 
&+ s_{(2)}s_{(3,1),\varnothing}(\x,\y)\\ 
&+ s_{\varnothing}s_{(4),\varnothing}(\x,\y)
\end{aligned}
\end{equation}

\subsection{Dimension of \texorpdfstring{$R_{\mathfrak{B}_n}^{(0,3)}$}{RBn(0,3)}}
From the computed Frobenius series, one can deduce that $\dim R_{\mathfrak{B}_1}^{(0,3)} = 4$, $\dim R_{\mathfrak{B}_2}^{(0,3)} = 21$, $\dim R_{\mathfrak{B}_3}^{(0,3)} = 121$, and $\dim R_{\mathfrak{B}_4}^{(0,3)} = 742$.

\bibliographystyle{plain}
\bibliography{main}

\begin{thebibliography}{10}

\bibitem{MR3629266}
Ron~M. Adin, Christos~A. Athanasiadis, Sergi Elizalde, and Yuval Roichman.
\newblock Character formulas and descents for the hyperoctahedral group.
\newblock {\em Adv. in Appl. Math.}, 87:128--169, 2017.

\bibitem{AjilaGriffeth}
Carlos Ajila and Stephen Griffeth.
\newblock Representation theory and the diagonal coinvariant ring of the type
  {B Weyl} group.
\newblock {\em Journal of Pure and Applied Algebra}, 228(6):107592, June 2024.

\bibitem{Angarone2025}
Robert Angarone, Patricia Commins, Trevor Karn, Satoshi Murai, and Brendon
  Rhoades.
\newblock Superspace coinvariants and hyperplane arrangements.
\newblock {\em Adv. Math.}, 467:36, 2025.

\bibitem{Artin}
Emil Artin.
\newblock {\em Galois theory. 2nd ed. {Edited} and supplemented with a section
  on applications by {Arthur} {N}. {Milgram}}, volume~2 of {\em Notre Dame
  Math. Lect.}
\newblock Univ. of Notre Dame Press, Notre Dame, IN, 1944.

\bibitem{BereleRegev}
A.~Berele and A.~Regev.
\newblock Hook {Young} diagrams with applications to combinatorics and to
  representations of {Lie} superalgebras.
\newblock {\em Adv. Math.}, 64:118--175, 1987.

\bibitem{Bergeron-code}
Fran{\c{c}}ois Bergeron.
\newblock Symmetric functions in {S}age.
\newblock Jupyter notebook.

\bibitem{Bergeron2013}
Fran{\c{c}}ois Bergeron.
\newblock Multivariate diagonal coinvariant spaces for complex reflection
  groups.
\newblock {\em Adv. Math.}, 239:97--108, 2013.

\bibitem{Bergeron2020}
Fran{\c{c}}ois Bergeron.
\newblock The bosonic-fermionic diagonal coinvariant modules conjecture.
\newblock {\em Preprint}, 2020.
\newblock \url{https://arxiv.org/abs/2005.00924}.

\bibitem{BergeronOPAC}
Fran{\c{c}}ois Bergeron.
\newblock $({GL}_k\times {S}_n)$-modules of multivariate diagonal harmonics.
\newblock In Christine Berkesch, Benjamin Brubaker, Gregg Musiker, Pavlo
  Pylyavskyy, and Victor Reiner, editors, {\em Open Problems in Algebraic
  Combinatorics}, volume 110 of {\em Proc. Symp. Pure Math.}, pages 1--22.
  Providence, RI: American Mathematical Society (AMS), 2024.

\bibitem{BergeronGarsia}
N.~Bergeron and A.~M. Garsia.
\newblock On certain spaces of harmonic polynomials.
\newblock In {\em Hypergeometric functions on domains of positivity, Jack
  polynomials, and applications. Proceedings of an AMS special session held
  March 22-23, 1991 in Tampa, FL, USA}, pages 51--86. Providence, RI: American
  Mathematical Society, 1992.

\bibitem{Bessenrodt}
Christine Bessenrodt.
\newblock Tensor products of representations of the symmetric groups and
  related groups.
\newblock {\em RIMS Kokyuroku (Proceedings of Research Institute for
  Mathematical Sciences)}, 1149:1--15, 2000.

\bibitem{bhattacharya}
Sutanay Bhattacharya.
\newblock The superspace coinvariant ring of type {B}.
\newblock {\em Preprint}, 2025.
\newblock \url{https://arxiv.org/abs/2505.24122}.

\bibitem{RhoadesBhattacharya}
Sutanay Bhattacharya and Brendon Rhoades.
\newblock Superspace coinvariants for wreath products.
\newblock {\em Preprint}, 2026.
\newblock \url{https://arxiv.org/abs/2606.30977}.

\bibitem{Borel}
Armand Borel.
\newblock Sur la cohomologie des espaces fibr{\'e}s principaux et des espaces
  homog{\`e}nes de groupes de {Lie} compacts.
\newblock {\em Ann. Math. (2)}, 57:115--207, 1953.

\bibitem{BourbakiAlgebraI}
Nicolas Bourbaki.
\newblock {\em Algebra {I}}.
\newblock Elements of mathematics. Springer-Verlag, Berlin, 1989.

\bibitem{ChengWangBook}
Shun-Jen Cheng and Weiqiang Wang.
\newblock {\em Dualities and representations of {L}ie superalgebras}, volume
  144 of {\em Graduate Studies in Mathematics}.
\newblock American Mathematical Society, Providence, RI, 2012.

\bibitem{Chevalley}
Claude Chevalley.
\newblock Invariants of finite groups generated by reflections.
\newblock {\em American Journal of Mathematics}, 77(4):778--782, 1955.

\bibitem{CurtisReiner}
Charles~W. Curtis and Irving Reiner.
\newblock {\em Representation theory of finite groups and associative
  algebras}.
\newblock AMS Chelsea Publishing, Providence, RI, 2006.
\newblock Reprint of the 1962 original.

\bibitem{GeckPfeiffer}
Meinolf Geck and G{\"o}tz Pfeiffer.
\newblock {\em Characters of finite {Coxeter} groups and {Iwahori}-{Hecke}
  algebras}, volume~21 of {\em Lond. Math. Soc. Monogr., New Ser.}
\newblock Oxford: Clarendon Press, 2000.

\bibitem{Gordon}
Iain Gordon.
\newblock On the quotient ring by diagonal invariants.
\newblock {\em Invent. Math.}, 153(3):503--518, 2003.

\bibitem{MR3839282}
Caroline Gruson and Vera Serganova.
\newblock {\em A journey through representation theory}.
\newblock Universitext. Springer, Cham, 2018.
\newblock From finite groups to quivers via algebras.

\bibitem{HaglundSergel}
James Haglund and Emily Sergel.
\newblock Schedules and the {Delta} conjecture.
\newblock {\em Annals of Combinatorics}, 25(1):1–31, November 2020.

\bibitem{Haiman1994}
Mark Haiman.
\newblock Conjectures on the quotient ring by diagonal invariants.
\newblock {\em J. Algebr. Comb.}, 3(1):17--76, 1994.

\bibitem{Haiman2002}
Mark Haiman.
\newblock Vanishing theorems and character formulas for the {Hilbert} scheme of
  points in the plane.
\newblock {\em Invent. Math.}, 149(2):371--407, 2002.

\bibitem{Howe}
Roger Howe.
\newblock Perspectives on invariant theory: {Schur} duality, multiplicity-free
  actions and beyond.
\newblock In {\em The Schur lectures (1992)}, pages 1--182. Ramat-Gan: Bar-Ilan
  University; Providence, RI: American Mathematical Society (Distrib.), 1995.

\bibitem{MR4381935}
Jongwon Kim and Brendon Rhoades.
\newblock Lefschetz theory for exterior algebras and fermionic diagonal
  coinvariants.
\newblock {\em Int. Math. Res. Not. IMRN}, 2022(4):2906--2933, 2022.

\bibitem{Lentfer-12}
John Lentfer.
\newblock A conjectural basis for the $(1,2)$-bosonic-fermionic coinvariant
  ring.
\newblock {\em Algebr. Comb.}, 8(3):711--743, 2025.

\bibitem{Lentfer-Supersymmetry}
John Lentfer.
\newblock Diagonal supersymmetry for coinvariant rings.
\newblock {\em Preprint}, 2025.
\newblock \url{https://arxiv.org/abs/2505.14885}.

\bibitem{Lentfer-dissertation}
John Lentfer.
\newblock {\em Combinatorics of Bosonic-Fermionic Coinvariant Rings}.
\newblock PhD thesis, University of California, Berkeley, 2026.

\bibitem{lentfer2026signcharactertriagonalfermionic}
John Lentfer.
\newblock The sign character of the triagonal fermionic coinvariant ring.
\newblock {\em Electron. J. Comb.}, 33(2):Paper P2.3, 22, 2026.

\bibitem{Leray}
Jean Leray.
\newblock Sur l'homologie des groupes de {Lie}, des espaces homog{\`e}nes et
  des espaces fibr{\'e}s principaux.
\newblock {\em Centre {Belge} {Rech}. {Math}., {Colloque} {Topologie},
  {Bruxelles}}, pages 101--115, 1951.

\bibitem{Macdonald}
Ian~Grant Macdonald.
\newblock {\em Symmetric functions and {Hall} polynomials.}
\newblock Oxford: Clarendon Press, 2nd ed. edition, 1995.

\bibitem{MR974302}
George~W. Mackey.
\newblock Induced representations and the applications of harmonic analysis.
\newblock In {\em Harmonic analysis ({L}uxembourg, 1987)}, volume 1359 of {\em
  Lecture Notes in Math.}, pages 16--51. Springer, Berlin, 1988.

\bibitem{fields2025}
Satoshi Murai, Brendon Rhoades, and Andy Wilson.
\newblock A proof of the {F}ields conjectures.
\newblock {\em Preprint}, 2025.
\newblock \url{https://arxiv.org/abs/2505.24027}.

\bibitem{OrellanaZabrocki}
Rosa Orellana and Michael Zabrocki.
\newblock A combinatorial model for the decomposition of multivariate
  polynomial rings as {{\(S_n\)}}-modules.
\newblock {\em Electron. J. Comb.}, 27(3):18, 2020.
\newblock Id/No p3.24.

\bibitem{Orellana}
Rosa~C. Orellana.
\newblock On the algebraic decomposition of a centralizer algebra of the
  hyperoctahedral group.
\newblock In {\em Algebraic structures and their representations. Proceedings
  of `XV coloquio Latinoamericano de \'algebra', Cocoyoc, Morelos, M\'exico,
  July 20--26, 2003.}, pages 345--357. Providence, RI: American Mathematical
  Society (AMS), 2005.

\bibitem{Pearson}
K.~Pearson.
\newblock On the theory of contingency and its relation to association and
  normal correlation.
\newblock London: {Draper}'s research memoirs. {Biometric} series. {Nrs}. 1-2,
  1904/5 (1904,1905)., 1904.

\bibitem{RhoadesWilson2023}
Brendon Rhoades and Andrew~Timothy Wilson.
\newblock The {Hilbert} series of the superspace coinvariant ring.
\newblock {\em Forum Math. Pi}, 12:35, 2024.
\newblock Id/No e16.

\bibitem{MR4674564}
Bruce~E. Sagan and Joshua~P. Swanson.
\newblock {$q$}-{S}tirling numbers in type {$B$}.
\newblock {\em European J. Combin.}, 118:Paper No. 103899, 35, 2024.

\bibitem{MR59914}
G.~C. Shephard and J.~A. Todd.
\newblock Finite unitary reflection groups.
\newblock {\em Canad. J. Math.}, 6:274--304, 1954.

\bibitem{MR526968}
Richard~P. Stanley.
\newblock Invariants of finite groups and their applications to combinatorics.
\newblock {\em Bull. Amer. Math. Soc. (N.S.)}, 1(3):475--511, 1979.

\bibitem{StanleyEC2}
Richard~P. Stanley.
\newblock {\em Enumerative combinatorics. {Volume} 2}, volume~62 of {\em Camb.
  Stud. Adv. Math.}
\newblock Cambridge: Cambridge University Press, 1999.

\bibitem{Stembridge}
John~R. Stembridge.
\newblock On the eigenvalues of representations of reflection groups and wreath
  products.
\newblock {\em Pacific J. Math.}, 140(2):353--396, 1989.

\bibitem{Stump}
Christian Stump.
\newblock $q,t$-{Fu\ss-Catalan} numbers for finite reflection groups.
\newblock {\em Journal of algebraic combinatorics}, 32(1):67--97, 2010.

\bibitem{SwansonPreprint}
Joshua~P. Swanson.
\newblock Alternating super-polynomials and super-coinvariants of finite
  reflection groups.
\newblock {\em preprint}, page \url{https://arxiv.org/pdf/1908.00196}, 2019.

\bibitem{SwansonWallach1}
Joshua~P. Swanson and Nolan~R. Wallach.
\newblock {Harmonic differential forms for pseudo-reflection groups I.
  Semi-invariants}.
\newblock {\em Journal of Combinatorial Theory, Series A}, 182:Paper no.
  105474, August 2021.

\bibitem{SwansonWallachII}
Joshua~P. Swanson and Nolan~R. Wallach.
\newblock Harmonic differential forms for pseudo-reflection groups. {II}:
  {Bi}-degree bounds.
\newblock {\em Comb. Theory}, 3(3):43, 2023.
\newblock Id/No 17.

\bibitem{Weyl}
Hermann Weyl.
\newblock {\em The classical groups, their invariants and representations},
  volume~1 of {\em Princeton Math. Ser.}
\newblock Princeton, NJ: Princeton University Press, 2nd edition, 1946.

\bibitem{Zabrocki-code}
Mike Zabrocki.
\newblock Jupyter notebook.
\newblock \url{http://garsia.math.yorku.ca/~zabrocki/delta_conj.ipynb}.

\bibitem{Zabrocki2019}
Mike Zabrocki.
\newblock A module for the {D}elta conjecture, 2019.
\newblock \url{https://arxiv.org/abs/1902.08966}.

\bibitem{Zabrocki2020}
Mike Zabrocki.
\newblock Coinvariants and harmonics, 2020.
\newblock
  \url{https://realopacblog.wordpress.com/2020/01/26/coinvariants-and-harmonics/}.

\end{thebibliography}
\end{document}